\documentclass[a4paper,12pt]{amsart}
\usepackage[utf8]{inputenc}

\usepackage[T1]{fontenc}
\usepackage[english]{babel}
\usepackage{color}
\usepackage{amsmath} 
\usepackage{amssymb} 
\usepackage{amsthm}
\usepackage{graphicx}
\usepackage[colorlinks=true, linkcolor=blue]{hyperref}

\usepackage{lmodern}
\usepackage{microtype}

\usepackage{soul}

\usepackage[colorinlistoftodos,prependcaption,textsize=tiny]{todonotes}

\theoremstyle{plain}

\newtheorem*{thm*}{Theorem}
\newtheorem{prop}{Proposition}[section]

\newtheorem*{cor*}{Corollary}

\newtheorem{defi}{Definition}[section]

\newcommand {\R} {\mathbb{R}} \newcommand {\Z} {\mathbb{Z}}
\newcommand {\T} {\mathbb{T}} 

\newcommand {\p} {\partial}
\newcommand {\dt} {\partial_t}

\newcommand{\bbT}{\mathbb{T}}

\newcommand{\bbZ}{\mathbb{Z}}

\newcommand{\calF}{\mathcal{F}}

\newcommand{\calR}{\mathcal{R}}

\newcommand{\eps}{\varepsilon}

\DeclareMathOperator{\supp }{supp}

\theoremstyle{plain}
\newtheorem{theorem}{Theorem}
\newtheorem{lemma}[theorem]{Lemma}
\newtheorem{pro}[theorem]{Proposition}

\theoremstyle{definition}

\date{\today}

\begin{document}

\title[MHD Stability Threshold]{On the Sobolev Stability Threshold for the 2D MHD Equations with
  Horizontal Magnetic Dissipation}

\begin{abstract}
  In this article we consider the stability threshold of the 2D
  magnetohydrodynamics (MHD) equations near a combination of Couette flow and
  large constant magnetic field.
  We study the partial dissipation regime with full viscous and only horizontal
  magnetic dissipation. In particular, we show that this regime behaves qualitatively differently than both the fully
  dissipative and the non-resistive setting.
\end{abstract}
\author{Niklas Knobel}
\address{Karlsruhe Institute of Technology, Englerstraße 2,
  76131 Karlsruhe, Germany}
  \email{niklas.knobel@kit.edu}
\author{Christian Zillinger}
\address{Karlsruhe Institute of Technology, Englerstraße 2,
  76131 Karlsruhe, Germany}
\email{christian.zillinger@kit.edu}


\maketitle

\setcounter{tocdepth}{1}
\tableofcontents

\section{Introduction }

The equations of magnetohydrodynamics (MHD)
\begin{align}
  \label{aniso}
  \begin{split}
    \partial_t V + V\cdot \nabla V+ \nabla \Pi  &= (\nu_x\partial_x^2+\nu_y\partial_y^2) V + B\cdot\nabla B, \\
    \partial_t B + V\cdot\nabla B &= (\kappa_x\partial_x^2+\kappa _y\partial_y^2)  B +B\cdot\nabla V, \\
    \nabla\cdot v=\nabla\cdot b  &= 0,\\
    (t,x,y) &\in \R^+ \times\bbT\times \R,
  \end{split}
\end{align}
model the evolution of the velocity $V$ of conducting, non-magnetic fluids
interacting with a magnetic field $B$. The MHD equations are commonly used in applications ranging from
astrophysics and the description of plasmas to control problems for liquid
metals in industrial applications \cite{davidson_2016}.
Similarly to the Navier-Stokes and Euler equations, questions of hydrodynamic
stability and the behavior for high Reynolds numbers (that is, for $\nu,\kappa$ tending to zero) are a very active area of
research both inner-mathematically and in view of applications.

Motivated by stability results for the isotropic full-dissipation case ($\nu_x=\nu_y=\kappa_x=\kappa_y>0$) and
instability results for the non-resistive case ($\kappa_x=\kappa_y=0$), we are
interested in the behavior of the two-dimensional magnetohydrodynamic (MHD)
equations with partial dissipation, where some of the dissipation coefficients
\begin{align*}
  \kappa_y, \kappa_x, \nu_x, \nu_y\ge 0,
\end{align*}
are allowed to vanish.
More specifically, we study the behavior near the stationary solution given by the combination of Couette flow and a
(large) constant magnetic field 
\begin{align}
  \label{eq:Couette}
    V_s=ye_1, \quad B_s=\alpha e_1,
\end{align}
for the case of vanishing vertical resistivity, $\kappa_y=0$.
For the related case of the Navier-Stokes equations (that is, without any
magnetic field) the (in)stability of Couette flow at high Reynolds number is
known as the Sommerfeld paradox \cite{Maj} and is related to nonlinear
instability of the Euler equations
\cite{bedrossian2015inviscid,dengmasmoudi2018,dengZ2019}.

However, for the case of sufficiently small data it was proven in
\cite{bedrossian2016sobolev} that (mixing enhanced) dissipation can counteract
this instability  in the Navier-Stokes equations and that (long time asymptotic) stability holds in Sobolev
spaces for initial data with
\begin{align*}
  \|\omega\|_{H^N}\leq \epsilon \ll \nu^{\gamma}
\end{align*}
with $\gamma\geq \frac{1}{2}$. Later in \cite{masmoudi2022stability} this has
been improved to $\gamma =\tfrac 13$.
This is an example of a stability threshold result, which establishes stability
for small data and determines suitable (optimal) exponents $\gamma$ for given norms.

Since the addition of the magnetic field is known to possibly destabilize the
dynamics (see the following discussion), our main questions concern the MHD
equations \eqref{aniso} in terms of perturbations moving with the underlying
shear flow:
\begin{align*}
    v(x,y,t)&= V(x-yt,y,t )- V_s, \\
    b(x,y,t)&= B(x-yt,y,t )- B_s.
\end{align*}
The corresponding perturbed equations in these new variables read
\begin{align}
\begin{split}
    \partial_t v + v_2 e_1 - 2\partial_x \Delta^{-1}_t  \nabla_t v_2  &=  \nu \cdot \Delta_t v+ \alpha \partial_x b  + b\nabla_t b- v\nabla_t v-\nabla_t \pi , \\
    \partial_t b - b_2 e_1 \qquad \qquad \quad \quad  \ &= \kappa\cdot \Delta_t   b+ \alpha \partial_x v  +b\nabla_t v -v\nabla_t b,\\
    \nabla_t\cdot v=\nabla_t\cdot b &= 0.\label{anco}
\end{split}
\end{align}
Here, we introduce the time-dependent derivatives $\partial_y^t = \partial_y-t\partial_x $, $\nabla_t = (\partial_x , \partial_y^t)$ and $\Delta _t = \partial_x^2+  (\partial_y^t)^2$.
Furthermore, we use the following short notation for the dissipation operator: 
\begin{align*}
    \nu \cdot \Delta_t &=\nu_x\partial_x^2+\nu_y(\partial_y^t)^2 ,\\
    \kappa \cdot \Delta_t &=\kappa_x\partial_x^2+\kappa_y(\partial_y^t)^2.
\end{align*}

In this article we aim to establish a Sobolev stability threshold for \eqref{anco} for the specific
anisotropic, partial dissipation case
\begin{align*}
  \kappa_y=0, \  \kappa_x = \nu_x =\nu_y>0.
\end{align*}
In particular, we show that this setting exhibits qualitatively different behavior than the fully dissipative and the non-resistive case. 

Following a similar notation as \cite{liss2020sobolev} we make the following definition.
\begin{defi}[Stability threshold]
  Consider the MHD equations \eqref{aniso} with anisotropic dissipation
  $0< \nu_x=\nu_y=\kappa_x=: \mu \ll 1$ and $\kappa_y=0$ and let $X$ be a Banach space with norm
  $\|(v,b)\|_{X}$.
  We then say that the exponent $\gamma=\gamma(X)$ is a stability threshold for
  the space $X$ if for initial data with
  \begin{align*}
    \Vert (v_{in},b_{in}) \Vert_X & \leq \epsilon \ll \mu^\gamma,
  \end{align*}
  the corresponding solution of \eqref{anco} remains uniformly bounded for all future times with
  a quantitative control
  \begin{align*}
    \sup_{t>0} \Vert (v,b) \Vert_X & \lesssim \epsilon.
  \end{align*}
\end{defi}
We remark that this definition does not require optimality (that is, instability
for smaller choices of $\gamma$). Optimal stability thresholds quantify the appearance of instability in the large
Reynolds number limit and are an active area of research for many fluid systems.
In view of the large literature, the interested reader is referred to the
following articles for the Navier-Stokes equations
\cite{bedrossian2016sobolev,bedrossian2017stability} and the Boussinesq
equations \cite{zhai2022stability,lai2021optimal,tao2020stability} for a discussion and further references.

For the (isotropic) MHD equations ($\nu:= \nu_x=\nu_y$ and $\kappa:= \kappa_x =\kappa_y$), there exists several results for non-vanishing magnetic dissipation. 
\begin{itemize}
    \item When considering full isotropic dissipation $\nu =\kappa>0$, Liss
      \cite{liss2020sobolev} established a Sobolev threshold in the 3D case.
      Under a Diophantine condition on the magnetic field, he establishes
      stability for $\|(v,b)\|_{H^N}$ with $\gamma= 1$.
      For the 2D case an improvement to $\gamma = \tfrac{2}{3}$ is expected
      due to the lack of lift-up instability.
      Indeed, in a very recent paper, \cite{Dolce}, Dolce establishes such a
      threshold for the regime $0<C\kappa^3 \leq \nu \leq \kappa$. 
    \item In the 2D inviscid case with isotropic magnetic dissipation, $\nu=0$ and
      $\kappa>0$, in \cite{knobel2023echoes} the authors established linear
      instability of nearby (in analytic regularity) so-called traveling wave
      type solutions in Gevrey $2$ regularity.
      As an (almost) matching nonlinear result, \cite{zhao2023asymptotic}
      established a stability threshold $\gamma\geq 1$ for Gevrey $2-\delta$
      regularity for any $0<\delta<1$.
    \item The setting with only an underlying magnetic field but without shear
      flow exhibits qualitatively different behavior and was studied for the
      case of the whole space in \cite{bardos1988longtime,ren2014global} in the
      full dissipation case and in \cite{cao20132d,ji2019stability} for the partially dissipative case.  
\end{itemize}
To the authors' knowledge there are no such results in the literature for the
non-resistive case $\kappa=0$ with Couette flow, both for the viscous or inviscid regime $\nu=0$ or
$\nu>0$, and neither for partial dissipation regimes.
In view of linear instability results \cite{hussain2018instability} (see also Proposition
\ref{prop:instability}), for these equations any stability threshold
results would need to consider unknowns different from $(v,b)$.

As a step towards understanding this non-resistive regime, in this article we
consider the 2D MHD equations with isotropic viscosity but only horizontal
resistivity (while \cite{liss2020sobolev, Dolce} consider full dissipation).
In particular, we ask to which extent, as quantified by Sobolev stability
thresholds, this partial dissipation regime behaves or does not behave like
these extremal cases.

In the (ideal) MHD equations ($\nu=\kappa=0$) the interaction of shear flows and the magnetic field has been shown to
possibly cause instabilities, with arguments both on physical
\cite{chen1991sufficient,hirota2005resonance} and mathematical grounds \cite{hughes2001instability, zhai2021long}. 

As our first result, we show that this instability also persists in the viscous
but non-resistive MHD. These equations exhibit norm inflation in $H^N$ for all
choices of $\nu>0$. 
\begin{pro}[Instability for the non-resistive MHD equations]
  \label{prop:instability}
  Consider the isotropic equation with $\nu>0$ and $\kappa=0$ and $N\ge 3$, then the stationary solution
  \eqref{eq:Couette} is linearly unstable in $H^N$.
  More precisely, there exists initial data $(v,b)_{in}\in H^N$ such that the
  solution to the linearized problem satisfies  
    \begin{align*}
        \Vert (v,b)\Vert_{H^N}\approx \tfrac \nu {\alpha^2} t\Vert (v,b)_{in} \Vert_{H^N}
    \end{align*}
    as $t\to \infty$.

    As a consequence, the nonlinear equations also exhibit arbitrarily large
    norm inflation in $H^N$. That is, for any $C=C(\nu) >0$ there exists an $\eps_0>0$
    such that for all $\eps<\eps_0$ there exists initial data $(v,b)_{in}$ and a
    time $T$ such that 
    \begin{align*}
        \Vert (v,b)_{in}  \Vert_{H^N} &=\eps ,\\
        \Vert (v,b)|_{t=T} \Vert_{H^N}&\ge C \Vert (v,b)_{in}  \Vert_{H^N}.
    \end{align*}
    In particular, there cannot exist a Sobolev threshold for
    $\|(v,b)\|_{H^N}$. 
\end{pro}
We remark that following the same argument also instability in suitable Gevrey spaces can be established.

As mentioned above, the isotropic fully dissipative case is known to be stable
in Sobolev regularity \cite{liss2020sobolev,Dolce}.
For the associated partial dissipation regimes, in view of the underlying shear
dynamics the associated vertical dissipation case is expected to behave
similarly as the full dissipation case.
The effects of partial dissipation are a very actively studied field of research in other fluid systems, such as the Boussinesq
equations \cite{deng2020stability, cao2013global,adhikari2022stability}), but, to the authors' knowledge, is largely open in the MHD equations near Couette flow.

In the present case of horizontal resistivity, $\kappa_y=0$ and $\nu_x=\nu_y = \kappa_x$,
the lack of vertical dissipation leads to stronger instabilities, requiring
finer control and use of the coupling by a strong magnetic field. Our main
results are summarized in the following theorem.
\begin{theorem}
  \label{thm:anisoThres}
Consider the MHD equations with horizontal resistivity,
$\mu:=\nu_x=\nu_y=\kappa_x>0$ and $\kappa_y=0$, near the stationary solution
\eqref{eq:Couette} with $\alpha >\tfrac 1 2  $ and let $N \ge 6$ be given.

Then there exist constants $c_0=c(\alpha ) >0$, such that for all initial data $(v,b) _{in }$ which satisfy 
    \begin{align*}
    \Vert   (v,b)_{in}  \Vert_{ H^N }=   \eps \le c_0  \mu^{\frac 32 }
\end{align*}
the corresponding solution $(v,b)$ of \eqref{eq_p} satisfies the estimates 
\begin{align*}
     \Vert  v \Vert_{L^\infty H^N }+\mu^{\frac 1 2 } \Vert  \nabla_t  v \Vert_{L^2 H^N } &\lesssim \eps, \\
     \Vert   b \Vert_{L^\infty H^N }+\mu^{\frac 1 2 } \Vert   \partial_x  b \Vert_{L^2H^N }&\lesssim \eps. 
\end{align*}
\end{theorem}

Let us comment on these results:
\begin{itemize}
\item Proposition \ref{prop:instability} shows instability in terms of $(v,b)$ for
  the non-resistive case. Hence, the (horizontal) magnetic dissipation is shown
  to be necessary for long-time stability results for $(v,b)$.
  
  However, similarly as in the Boussinesq equations
  \cite{bedrossian21,zillinger2021echo}, in principle stability results in terms
  of other unknowns such as the magnetic potential $\phi =
  (-\Delta_t)^{-1}\nabla^\perp_t b$ could hold for longer or even infinite
  times, which remains an exciting question for future research.
\item Theorem \ref{thm:anisoThres} establishes a stability threshold
  $\gamma=\frac{3}{2}$. In particular, we stress that the lack of vertical magnetic dissipation not only poses a key challenge of our analysis but results in a different threshold value than the fully dissipative setting \cite{liss2020sobolev, Dolce}.
  
  Indeed, the main constraint on our stability threshold is given by the control of the nonlinearity $v \cdot \nabla_t b$ and the reduced decay rates already at the linearized level (see Section \ref{linstab}).
  As we show in Section \ref{hfw}, our estimates of the so-called reaction terms \eqref{est:bvbnR} and \eqref{est:bvbaR} require a lower bound on the threshold by $\tfrac 3 2 $ and are expected to be optimal for this partial dissipation case. 
\item  For simplicity of notation, in Theorem \ref{thm:anisoThres} we have
  stated our result for the case $\mu := \nu_x=\nu_y=\kappa_x$.
  As we discuss in Sections \ref{linstab} and \ref{bootHyp}, more generally, instead of equality it suffices to require that $\frac{1}{2\alpha}\nu_y\le \kappa_x\le C \nu_y^{\frac 1 3}$, similarly as in the full dissipation case studied in \cite{Dolce}. Furthermore, we expect that results can be be extended to the case of purely vertical viscous dissipation with additional technical effort.   
  \item Due to missing vertical dissipation, we obtain no decay of the $x$-averaged magnetic field $b_=$ which is forced by the nonlinearity.
\end{itemize}
To prove our results, it is convenient to work with the unknowns
\begin{align*}
  p_1= \Lambda^{-1}_t \nabla^\perp_t \cdot  v, \ p_2= \Lambda^{-1}_t \nabla^\perp_t \cdot b; \quad \Lambda_t:= \sqrt{-\Delta_t}. 
\end{align*}
Similarly to the vorticity and current, the curl operator $\nabla_t^{\perp}$
eliminates the pressure and yields a scalar quantity, while the
operator $ \Lambda^{-1}_t \nabla^\perp_t \cdot $ is of order $0$.
Moreover, since  $v$ and $b$ are divergence-free, similarly to viscosity
formulations of the 2D Navier-Stokes equations, it can be shown by integration
by parts that 
\begin{align*}
\begin{split}
  \Vert A v\Vert_ {L^2} &= \Vert A p_1 \Vert_{L^2}, \\
  \Vert A b\Vert_ {L^2} &= \Vert A p_2 \Vert_{L^2},    
\end{split}
\end{align*}
for all Fourier multiplier $A$ which commute with $\nabla_t$ and $\Lambda_t$. This, in particular, includes $\langle \nabla\rangle^N$ which corresponds to the Sobolev norm $\| \cdot \|_{H^N}$.

In terms of these unknowns our equations read
\begin{align}
  \label{eq_p}
\begin{split}
     \partial_t p_1 - \partial_x \partial_x^t \Delta^{-1}_t p_1- \alpha \partial_x p_2 &= \nu \cdot \Delta_t p_1 +\Lambda^{-1}_t  \nabla^\perp_t (b\nabla_t b- v\nabla_t v), \\
  \partial_t p_2 +\partial_x \partial_x^t \Delta^{-1}_t p_2 - \alpha \partial_x p_1 &= \kappa \cdot \Delta_t p_2  +\Lambda^{-1}_t \nabla^\perp_t (b\nabla_t v- v\nabla_t b), \\
  \nu= (\mu, \mu), \ \kappa &= (\mu, 0).
\end{split}
\end{align}

The remainder of the article is structured as follows:
\begin{itemize}
\item In Section \ref{linstab}, as a first step we establish linear stability of the
  equations \eqref{eq_p}. In view of the lack of vertical resistivity we here crucially rely
  on the interaction of $p_1$ and $p_2$ due the the underlying constant magnetic field. Moreover, we discuss the effects of partial dissipation and the resulting limited (optimal) decay rates in time.
\item In Section \ref{bootHyp}, we introduce a bootstrap method for the proof of
  Theorem \ref{thm:anisoThres}. Decomposing into low and high frequency
  contributions here yields several error terms, which are handled in different
  subsections.
  In particular, we need to distinguish between the evolution of the $x$-average
  (which does not experience enhanced dissipation due to the shear) and its
  $L^2$-orthogonal complement, as well as different frequency decompositions of the
  nonlinear terms (called reaction and transport terms in the literature). 
\item More precisely, in Subsection \ref{dnl} we collect all nonlinear terms
  which can be estimated in a straightforward way.
 In view of partial magnetic dissipation a main challenge is given by the effect
 of $v\nabla_t b$ on $p_2$ at high frequencies. Here, we distinguish between
 terms without $x$-average in Subsection \ref{hfw} and with average in
  Subsection \ref{hfa} and perform a decomposition into a transport and a reaction term.
 The low frequency regime is discussed in Subsection \ref{lf} and does not require a very precise analysis.
\item As a complementary result, in Section \ref{instab} we establish instability of the non-resistive, viscous MHD equations and prove Proposition
  \ref{prop:instability}. Here we first prove linear algebraic instability and then deduce a nonlinear norm inflation result as a corollary.
\end{itemize}

\subsection{Notations and conventions}
\label{sec:notation}
For two real numbers $a,b \in \R$ we denote the minimum and maximum as 
\begin{align*}
    \min(a,b)&=a\wedge b,\\
    \max(a,b)&=a\vee b.
\end{align*}
We use the notation $f\lesssim g$ if there exists a constant $C$ independent of all relevant parameters such that $|f|\le C |g| $. Furthermore, we write $f\approx g$ if $f\lesssim g$ and $g\lesssim f $.

Moreover, for any vector or scalar $v$ we define 
\begin{align*}
    \langle v \rangle&= (1+\vert v\vert^2)^{\frac 1 2 }.
\end{align*}
For a function $f(x,y) \in L^2(\T\times\R)$ we denote the $x$-average and its $L^2$-orthogonal complement as
\begin{align*}
    f_=(y) &= \int_\bbT f(x,y) dx,\\
    f_{\neq }&= f-f_{0}.
\end{align*}

Throughout this text, unless noted otherwise, the spatial variables $(x,y)\in \T \times \R$ are periodic in the horizontal direction and the respective Fourier variables are denoted as
\begin{align*}
    (k,\xi)\in (\bbZ, \R )
\end{align*}
or $(l,\eta)$. The norms $\Vert \cdot \Vert_{L^p}$ and $\Vert \cdot \Vert_{H^N}$ refer to the standard Lebesgue and Sobolev norms for functions on $\T \times \R$.
For time-dependent functions we denote $L^p H^s=L^p_t H^s$ as the space with the norm 
\begin{align*}
    \Vert f \Vert_{L^pH^s}&= \left\Vert \Vert f\Vert_{H^s(\T\times \R)}\right \Vert_{L^p(0,T)}.
\end{align*}
We define the weight $A^N$ and $A^{N'}_\mu$ by the Fourier multipliers 
\begin{align*}
    A^N&= M \langle \nabla \rangle^N, \\
     A^{N'}_\mu  &= M \langle \nabla \rangle^{N'} e^{c\mu t\textbf{1}_{k\neq 0 }},
\end{align*}
for $3<N' \le N-2$ and $0<c< \tfrac 12(1-\sqrt{\tfrac 2 3 })$. With slight abuse of notation we identify the multiplier operators with their Fourier symbols. The operator $M$ is a time dependent Fourier multiplier, introduced in \cite {bedrossian2016sobolev}, and is defined to satisfy the following equation:
\begin{align*}
    -\tfrac {\dot M }{M }&= \tfrac {\vert k \vert } {k^2 + \vert \xi -kt\vert^2},\\
    M(0,k,\xi)&=1.
\end{align*}
That is, $M$ is given as
\begin{align*}
    M(t,k,\xi)&= \exp\left( - \int_0^t d\tau\tfrac {\vert k \vert }{k^2+ (\xi-k\tau )^2}\right).
\end{align*}
In particular, the operator $M$ is comparable to the identity in the sense that
\begin{align*}
    1\ge M(t,k,\xi) \ge c
\end{align*}
for some constant $c$ and all $k\neq 0$ (and $M(t,0,\xi):=1$ for $k=0$).

The operators $A$ thus define energies comparable to Sobolev (semi)norms:
\begin{align*}
    \Vert A^N \cdot \Vert_{L^2} &\approx \Vert\cdot\Vert_{H^N},\\
    \Vert A^{N'}_\mu  \cdot \Vert_{L^2} &\approx \Vert e^{c\mu t\textbf{1}_{k\neq 0}   } \cdot \Vert_{H^N}.
\end{align*}
In particular, since $N$ is sufficiently large, the norm defined by $A^N$ satisfies an algebra property.

\section{Linear stability } \label{linstab}
In this section we study the stability of the linearized version of the
equations \eqref{eq_p}:
\begin{align}
  \label{eq_p_lin}
\begin{split}
     \partial_t p_1 - \partial_x \partial_x^t \Delta^{-1}_t p_1- \alpha \partial_x p_2 &= \nu \cdot \Delta_t p_1,\\
  \partial_t p_2 +\partial_x \partial_x^t \Delta^{-1}_t p_2 - \alpha \partial_x p_1 &= \kappa \cdot \Delta_t p_2,\\
  \nu= (\mu, \mu), \ \kappa &= (\mu, 0).
\end{split}
\end{align}
The ode tools to establish stability of such systems are well-known in related
systems such as the Boussinesq equations \cite{lai2021optimal,bedrossian21,bianchini2020linear, masmoudi2023asymptotic,zillinger2020boussinesq}.

Our main results are summarized in the following proposition.
\begin{prop}[Linear stability]
  \label{prop:lin_stability}
  Let $\mu>0$, $\alpha> \tfrac 1 2 $ and $N\geq 6$ be as in Theorem
  \ref{thm:anisoThres}. Then the equations \eqref{eq_p_lin} are stable in $H^N$
  in the sense that for any choice of
  initial data $p_{in} \in H^N$ the corresponding solution satisfies
  \begin{align*}
    \|p\|_{L^\infty H^N} + \mu^{1/2} \|\nabla_t p_1\|_{L^2 H^N} + \mu^{1/2} \|\p_x p_2\|_{L^2 H^N} \lesssim e^{-C\mu t} \|p_{in}\|_{H^N}.  
  \end{align*}
\end{prop}
As we discuss in the proof, in the case $\nu \leq \kappa \leq \nu^{3}$ the optimal decay rate for large times is given by $\mu=\min(\nu^{1/3},\kappa)$. In particular, the coupling induced by the underlying magnetic field cannot yield enhanced dissipation rates for both components once the viscous dissipation becomes too large.

\begin{proof}[Proof of Proposition \ref{prop:lin_stability}]
  We note that in this linear evolution equation \eqref{eq_p_lin} all
  coefficient functions are independent of both $x$ and $y$.
  Therefore the equations decouple after a Fourier transform and we may equivalently
  consider the ode system
  \begin{align}
    \begin{split}
      \partial_t \hat  p_1 - \tfrac {k (\xi-kt )}{k^2+(\xi-kt)^2} \hat p_1- \alpha ik\hat  p_2 &=-\nu (  k^2+(\xi-kt)^2)\hat  p_1,  \\
      \partial_t\hat  p_2 +\tfrac {k (\xi-kt )}{k^2+(\xi-kt)^2}\hat  p_2 - \alpha ik\hat p_1 &= -\kappa k^2\hat  p_2,
    \end{split}
  \end{align}
  for an arbitrary but fixed frequency $(k,\eta) \in \Z \times \R$.
  Since the equations are trivial for $k=0$, in the following we further without
  loss of generality may assume that $k\neq 0$. Furthermore, after shifting $t$
  by $\frac{\xi}{k}$, we may assume that $\xi=0$ and thus obtain a system of the
  form
  \begin{align}
  \label{eq:linodesystem}
  \begin{split}
    \dt
    \begin{pmatrix}
      p_1 \\ p_2
    \end{pmatrix}
    =
    \begin{pmatrix}
      -\frac{t}{1+t^2}- \nu k^2 (1+t^2) & i\alpha k \\
      i\alpha k &  \frac{t}{1+t^2} - \kappa k^2
    \end{pmatrix}
    \begin{pmatrix}
      p_1\\ p_2
    \end{pmatrix},  
  \end{split}
  \end{align}
  where we dropped the hats for simplicity of notation.
  
In a first naive estimate, we can test this equations with $(p_1,p_2)$ and obtain
that
\begin{align*}
  \dt (|p_1|^2 + |p_2|^2) \leq (\frac{|t|}{1+t^2} - \mu k^2) (|p_1|^2 + |p_2|^2),
\end{align*}
which already yields the desired decay for times $|t|\gg (\mu k^2)^{-1}$.
However, a Gronwall-type estimate on the remaining interval would only yield a
very rough upper bound on the possible growth by
\begin{align*}
  \exp\left( \int_{|t|\lesssim (\mu k^2)^{-1}}\frac{|t|}{1+t^2} dt \right) \lesssim (1+\mu k^2)^2.
\end{align*}

In order to improve this estimate, a common trick is to make use of the fact
that $|\alpha|$ is relatively large and to consider
\begin{align*}
  E= |p_1|^2 + |p_2|^2 + \frac{t}{1+t^2} \Re \left( \frac{1}{i\alpha} p_1 \overline{p_2} \right).
\end{align*}
Since $|\alpha|>\frac{1}{2}$ this energy is positive definite and comparable to
$|p_1|^2 +|p_2|^2$, with a constant which degenerates as $|\alpha| \downarrow \frac{1}{2}$.

Computing the time derivative of the last term, we note that
\begin{align*}
  & \quad \frac{t}{1+t^2} \p_t \Re \left( \frac{1}{i\alpha} p_1 \overline{p_2} \right)\\
  & \leq  \frac{t}{1+t^2} (|p_1|^2-|p_2|^2) \\
  & \quad + \frac{|t|}{1+t^2} \frac{1}{|\alpha|} \nu k^2 (1+t^2) |p_1||p_2| \\
  & \quad + \frac{|t|}{1+t^2} \frac{1}{|\alpha|} \kappa k^2 |p_1||p_2| \\
  & \quad + \mathcal{O}(t^{-2})|p_1||p_2|.
\end{align*}
The first term exactly cancels out the possibly large contribution in $\dt
(|p_1|^2+|p_2|^2)$. For the second and third term we use that fact that
$\frac{1}{\alpha} < 2$ and that these terms can hence be absorbed into the
dissipation terms at the cost of a slight loss of constants, provided that 
\begin{align*}
    \frac{1}{2\alpha}\nu \leq \kappa \leq \nu^{1/3}.
\end{align*}
Noting that $\dt \frac{|t|}{1+t^2}= \mathcal{O}(t^{-2})$ is integrable
in time, we thus arrive at
\begin{align*}
  \dt E \lesssim \mathcal{O}(t^{-2})E - \nu k^2 (1+t^2)|p_1|^2 - \kappa k^2 |p_2|^2.
\end{align*}
Further defining
\begin{align*}
  \tilde{E} = E \exp(\int^t \mathcal{O}(\tau^{-2}) d\tau),
\end{align*}
it follows that $\tilde{E}\approx E$ decays exponentially in time and that the damping terms are integrable in time, which yields the desired result.

We further remark that for $t$ (corresponding to times $t+\frac{\xi}{k}$) such that $|t|\lesssim (\mu k^2)^{-1/3}$ the system \eqref{eq:linodesystem} exhibits mixing enhanced dissipation, even though the dissipation for the magnetic component is only horizontal.
Indeed, after relabeling $p_1 \mapsto i p_1$ and introducing the energy $E$ to control contributions by $\frac{t}{1+t^2}$, this follows from the fact that the eigenvalues of the matrix
\begin{align*}
    \begin{pmatrix}
        -\mu k^2 (1+t^2) & -\alpha \\
        \alpha & -\mu k^2
    \end{pmatrix}
\end{align*}
are given by 
\begin{align*}
    \lambda_{1,2} = - \frac{\mu k^2 (2+t ^2)}{2} \pm \sqrt{\frac{1}{4}(\mu k^2 t^2)^2 - \alpha^2}.
\end{align*}
Since $|\alpha|>\frac{1}{2}$ by our assumption on $t$ the square root is purely imaginary and hence $\Re(\lambda_1)=\Re(\lambda_2)$ is comparable to the enhanced dissipation term 
\begin{align*}
    - \mu k^2 (1+t^2).
\end{align*}

However, for times much larger than this (that is, far away from $\frac{\xi}{k}$), the same eigenvalue computation shows that 
\begin{align*}
    \lambda_1 \approx -\mu k^2 \langle t\rangle^2, \ \lambda_2 \approx -\mu k^2
\end{align*}
and hence enhanced dissipation can only be expected for one of the eigenvalues.
\end{proof}
This linear result highlights the effects of the coupling induced by the underlying constant magnetic field and shows which optimal decay estimates can be expected.
In particular, it clearly illustrates that the loss of vertical magnetic dissipation incurs a change of decay rate compared to the fully dissipative case.

\section{Bootstrap hypotheses and outline of proof}
\label{bootHyp}
We next turn to the full nonlinear problem \eqref{eq_p}, where we intend to
treat the nonlinear contributions as errors and make use of the smallness of our
initial data.

Our approach here follows a bootstrap argument, which is by now standard in the
field (see, for instance, \cite{bedrossian2016sobolev}).
In the notation of Section~\ref{sec:notation} we assume that at the initial time
\begin{align}
  \label{eq:initialassump}
  \|A^N p\|_{L^2}^2 + \|A_{\mu}^{N'}p\|_{L^2}^2 \le c_0  \epsilon^2
\end{align}
 for $3<N' \le N-2$. The constant $c_0=c_0(\alpha)>0$ will later be chosen small enough and tends to $0$ as $\alpha \to \tfrac 1 2$.  Given this estimate at the initial time, our aim in the remainder of this section is to establish the following estimates for the corresponding solution:
 
\begin{itemize}
\item 
\textbf{High frequency estimates} 
\begin{align}
  \label{eq:hfassump}
  \begin{split}
     \Vert A^N  p_1 \Vert^2_{L^\infty L^2 }+\mu \Vert A ^N \nabla_t p_1 \Vert^2_{L^2 L^2 }+ \Vert \sqrt{-\tfrac {\dot M} {M}  } A^N   p_1 \Vert^2_{L^2L^2 } &<\eps^2, \\
     \Vert A^N  p_2 \Vert^2_{L^\infty L^2 }+\mu \Vert A^N  \partial_x  p_2 \Vert^2_{L^2L^2 }+\Vert\sqrt{  -\tfrac {\dot M} {M} } A^N   p_2 \Vert^2_{L^2L^2 } &<\eps^2.
  \end{split}
\end{align}
\item  \textbf{Low frequency estimates} 
  \begin{align}
    \label{eq:lfassump}
    \begin{split}
     \Vert A^{N'}_\mu p_1 \Vert^2_{L^\infty L^2 }+\mu \Vert A ^{N'}_\mu \nabla_t p_1 \Vert^2_{L^2 L^2 }+\Vert \sqrt{-\tfrac {\dot M} {M}  } A^{N'}_\mu   p_1 \Vert^2_{L^2L^2 } &<\eps^2, \\
     \Vert A^{N'}_\mu  p_2 \Vert^2_{L^\infty L^2 }+\mu \Vert A^{N'}_\mu  \partial_x  p_2 \Vert^2_{L^2L^2 } +\Vert\sqrt{-  \tfrac {\dot M} {M} } A^{N'}_\mu   p_2 \Vert^2_{L^2L^2 }&<\eps^2.
    \end{split}
\end{align}
\end{itemize}
By local well-posedness and our assumptions on the initial data, these estimates
are satisfied at least on some (small) time interval $(0,T)$. In our bootstrap approach we assume for the sake of contradiction that the maximal time $T$ with this property is finite. We then show that on that same time interval all estimates hold with improved bounds instead, which however would imply that the estimates could be extended for a small additional time, contradicting the
maximality of $T$.

With this understanding, we suppress $T$ in our notation (see Section \ref{sec:notation}) and all $L^p$ norms in time are understood to be norms on $L^p(0,T)$.

The splitting into high and low frequencies is essential to close the estimates in Subsection \ref{hfw} and Subsection \ref{hfa}. In particular, we need the $e^{-c\mu t }$ decay to bound the so-called reaction error. 
Moreover, we require strong control of commutators involving $A$ in order to control the so-called transport error.
Both error terms impose strong restrictions on the energies and do not allow to close estimates in an easy way.
We overcome this difficulty by linking separate energy estimates in the high frequency part and the low frequency part. On the one hand, we can use the additional $e^{-c\mu t } $ in the low frequency part to our benefit in the analysis of the high frequency part. On the other hand, the difference in regularity allows us to control derivatives in the low frequency estimate by the using high frequency estimate.

Given a solution $(p_1,p_2)$ of \eqref{eq_p} and letting $A= A^N, A^{N'}_\mu$,
computing time derivatives we need to control
\begin{align*}
      &\quad  \p_t \Vert A p_1 \Vert^2_{L^2 }+2(1-c)\mu \Vert A  \nabla_t  p_1 \Vert^2_{L^2 }+2\Vert \sqrt{-\tfrac {\dot M} {M}  } A   p_1 \Vert^2_{L^2 }\\
     &\le 2 \langle A^2 p_1, \partial_x \partial_x^t \Delta^{-1}_t p_1 +\Lambda^{-1}_t  \nabla^\perp_t (b\nabla_t b- v\nabla_t v)\rangle =: L[p_1]+NL[p_1], \\
    &\quad   \p_t\Vert A  p_2 \Vert^2_{ L^2 }+ 2(1-c)\mu \Vert A   \partial_x  p_2 \Vert^2_{L^2 }+2 \Vert\sqrt{  -\tfrac {\dot M} {M} } A   p_2 \Vert^2_{L^2 }\\
    &\le 2 \langle A^2 p_2, -\partial_x \partial_x^t \Delta^{-1}_t p_2 +\Lambda^{-1}_t \nabla^\perp_t (b\nabla_t v- v\nabla_t b)\rangle=: L[p_2]+NL[p_2].
\end{align*}
Here we have split contributions into linear (that is, quadratic integrals) and nonlinear terms (that is, trilinear integrals). 
Note that the choice of $0<c< \tfrac 12(1-\sqrt{\tfrac 2 3 })$ is made such that $1-c$ is not too small to absorb linear effects for $\alpha$ close to $\tfrac 1 2$. For later reference, we note the following estimates:
\begin{align}
    \Vert \partial_x^2 \Lambda_t^{-1} \Lambda^{-1} p\Vert_{H^N }\lesssim \tfrac 1 t \Vert  p_{\neq}\Vert_{H^N }\label{damp}
\end{align}
and for $A= A^N, A^{N'}_\mu$
\begin{align}
  \label{eq:simpleenergy}
  \begin{split}
    \Vert A p_{1,\neq} \Vert_{L^2 L^2 }&\lesssim \mu^{-\frac 1 2 } \eps,\\
    \Vert A p_{2,\neq} \Vert_{L^2 L^2 }&\lesssim \mu^{-\frac12} \eps.
  \end{split}
\end{align}
Furthermore, for the nonlinear terms we will often use the equality 
\begin{align*}
\begin{split}
  \Vert A v\Vert_ {L^2} &= \Vert A p_1 \Vert_{L^2}, \\
  \Vert A b\Vert_ {L^2} &= \Vert A p_2 \Vert_{L^2}.
\end{split}
\end{align*}

Throughout the following sections, we aim to establish smallness of the
contributions by the linear terms $L[\cdot]$ and nonlinear terms $NL[\cdot]$. More precisely, we
establish the following proposition.
\begin{prop}[Control of errors]
\label{prop:errors}
  Under the assumptions of Theorem \ref{thm:anisoThres} suppose that the initial
 data satisfies the smallness condition \eqref{eq:initialassump} and let $T>0$
 be such the high and low frequency estimates \eqref{eq:hfassump},
 \eqref{eq:lfassump} are satisfied.
 Then on the same time interval it holds that
 \begin{align*}
   \int_0^T L[p_1]+ L[p_2] dt &\leq \tfrac{1}{2\alpha}  (c_0+ 1)\eps^2 + O(\mu^{-1}\eps^3) ,\\
   \int_0^T  NL[p_1] +NL[p_2] dt & \leq \mu^{-\frac 3 2 } \eps^3.
 \end{align*}
\end{prop}
As a consequence, supposing that $\alpha>\frac 1 2$ and $\epsilon\ll \mu^{3/2}$, this implies that both the high frequency and low frequency estimates \eqref{eq:hfassump}, \eqref{eq:lfassump} improve and
thus $T$ can only have been maximal if $T=\infty$, which proves Theorem \ref{thm:anisoThres}. Thus proving Proposition \ref{prop:errors} is our main concern in this section and our proof is split over the following subsections. The most important estimates, highlighting the effects of partial dissipation, are established in Subsections \ref{sec:linerror}, \ref{hfw} and \ref{hfa}. 

We note that the nonlinear terms
\begin{align*}
    \langle A  p_1 , \Lambda^{-1}_t  \nabla^\perp_t A  (b\nabla_t b- v\nabla_t v)\rangle &=-\langle A  v , A  (b\nabla_t b- v\nabla_t v)\rangle,\\
    \langle A  p_2 , \Lambda^{-1}_t  \nabla^\perp_t A  (b\nabla_t v- v\nabla_t b)\rangle&=-\langle A  b  ,  A (b\nabla_t v- v\nabla_t b)\rangle,
\end{align*}
for $A= A^N, A^{N'}_\mu $ are all trilinear products involving
\begin{align*}
  a^1a^2a^3\in \{vvv,vbb,bbv,bvb\}
\end{align*}
and we will use this notation to refer to the specific terms. Since the $x$-averages do not experience fast (mixing enhanced) decay under the
dissipation, we split these products as 
\begin{align*}
    \langle A a^1, A(a^2\nabla_t a^3)\rangle &= \langle A a^1_{\neq} , A(a^2_{\neq}\nabla_t a^3_{\neq})_{\neq} \rangle\\
    & \quad +\langle A a^1_{\neq} , A(a^2_=\nabla_t a^3_{\neq}) \rangle\\
    &\quad +\langle A a^1_{\neq} , A(a^2_{\neq}\nabla_t a^3_=) \rangle\\
    &\quad + \langle A a^1_{=} , A(a^2_{\neq}\nabla_t a^3_{\neq})_=,
\end{align*}
where the full splitting is only used for the $bvb$ term.

\subsection{Estimate of the linear error}
\label{sec:linerror}
In this subsection we establish the estimate of the linear terms in Proposition
\ref{prop:errors}. Here, we use some of the same techniques as in the proof of
linear stability in Section \ref{linstab}, but instead focus on establishing
quantitative bounds on the time integral.

Taking a Fourier transform of \eqref{eq_p} yields
\begin{align}
\begin{split}
     \partial_t \hat  p_1 - \tfrac {k (\xi-kt )}{k^2+(\xi-kt)^2} \hat p_1- \alpha ik\hat  p_2 &=-\mu (  k^2+(\xi-kt)^2)\hat  p_1 +\calF[NL[p_1]],  \\
      \partial_t\hat  p_2 +\tfrac {k (\xi-kt )}{k^2+(\xi-kt)^2}\hat  p_2 - \alpha ik\hat p_1 &= -                      \mu k^2\hat  p_2  +\calF[NL[p_2]]\label{eq_pf}.
\end{split}
\end{align}
Recalling the various contributions, we aim to estimate 
\begin{align*}
    & \quad \langle A^2 p_2 , -\partial_x \partial_y^t \Delta^{-1}_t p_2 \rangle + \langle A^2 p_1 , \partial_x \partial_y^t \Delta^{-1}_t p_1\rangle\\
    &= \sum_k \int d\xi A^2 \tfrac {k (\xi-kt )}{k^2+(\xi-kt)^2}(\vert \hat p_1 \vert^2-\vert \hat p_2 \vert^2 ) .
\end{align*}
In the following, with slight abuse of notation, we omit the hat denoting the Fourier transform and only consider $k\neq 0$, since for $k=0$ this term vanishes. 

Similarly as in the linear stability results of Section \ref{linstab}, we note that the Fourier multiplier a priori is not integrable in time and cannot easily be estimated by the partial dissipation. Hence, we rely on the coupling induced by the underlying magnetic field to eliminate some of this contribution and to provide better decay.
More precisely, multiplying the equations \eqref{eq_pf}  with $\hat{p}_2, \hat{p}_1$ and omitting the hats for simplicity of notation, we obtain the following identity:
\begin{align*}
  & \quad \vert  p_1 (k) \vert^2-\vert  p_2 (k)\vert^2\\
  &= -\tfrac 1 {i \alpha k }( p_1 \overline { i \alpha k p_1 } +i \alpha k p_2 \overline {  p_2} )\\
    &=-\tfrac 1 {i \alpha k } p_1 (\partial_t  \overline{p}_2 +\tfrac {k (\xi-kt )}{k^2+(\xi-kt)^2}\  \overline{p}_2 + \mu k^2  \overline{p}_2  -\overline{\calF} [{NL}[p_2]]))\\
    & \quad -\tfrac 1 {i \alpha k } \overline p_2  (\partial_t  p_1 -\tfrac {k (\xi-kt )}{k^2+(\xi-kt)^2}  p_1+(\mu k^2+\mu (\xi-kt)^2)  p_1  -\calF[NL[p_1]]\\    &= \tfrac {-1}{i\alpha k } (\partial_t (p_1  \overline  p_2)+ \mu (k^2+(\xi-kt)^2)p_1 \overline  p_2+ \mu k^2p_1 \overline  p_2) \\
    & \quad -\tfrac 1 {\alpha ik } ( p_1, p_2) \cdot \calF[\Lambda^{-1}_t \nabla_t^\perp (b\nabla_t b -v\nabla_t v , b\nabla_t v-v\nabla_t b)].
\end{align*}
Thus we split $L$ into two linear terms and one nonlinear term:
\begin{align}
  \label{eq:ONL}
  \begin{split}
    L&=2\sum_k \int d\xi A^2 \tfrac {k (\xi-kt )}{k^2+(\xi-kt)^2}\tfrac {-1}{i\alpha k } \partial_t (p_1  \overline  p_2) \\
    & \quad + 2\sum_k \int d\xi A^2 \tfrac {k (\xi-kt )}{k^2+(\xi-kt)^2}\tfrac {-1}{i\alpha k }(2 \mu k^2+\mu (\xi-kt)^2)p_1 \overline  p_2 \\
    & \quad -\tfrac 2 {\alpha }\langle  A  \partial_y^t \Delta_t^{-1}  (p_1,p_2)_{\neq} ,A \Lambda^{-1}_t \nabla_t^\perp (b\nabla_t b -v\nabla_t v , b\nabla_t v-v\nabla_t b)_{\neq}\rangle \\
    &=L_1 +L_{\mu}+ONL.
  \end{split}
\end{align}
We estimate $L_{\mu}$ by 
\begin{align*}
    L_{\mu}&= \tfrac 2\alpha \mu \sum_{k \neq 0 }\int d\xi A^2 \tfrac {(2 k^2+(\xi-kt )^2 )(\xi-kt )}{k^2+(\xi-kt)^2}p_1 \overline  p_2 \\
    &= \tfrac 2\alpha \mu \sum_{k \neq 0 }\int d\xi A^2 \tfrac {(2 k^2+(\xi-kt )^2 )(\xi-kt )}{(k^2+(\xi-kt)^2)^{\frac 3 2}}  p_1 (k^2+(\xi-kt)^2)^{\frac 1 2 } \overline  p_2 \\
    &\le \tfrac 2 \alpha  \mu \sup_s \left( \tfrac {(2+s^2)s}{(1+s^2)^{\frac 3 2 }} \right)  \Vert A \partial_x p_2 \Vert_{L^2} \Vert A \nabla_t p_1 \Vert_{L^2} \\
    &\le \sqrt{\tfrac 23 } \tfrac 1 \alpha  \mu   (\Vert A \partial_x p_2 \Vert_{L^2}^2 +\Vert A \nabla_t p_1 \Vert_{L^2}^2),
\end{align*}
where we used that 
\begin{align*}
    \left \vert \tfrac {(2 k^2+(\xi-kt )^2 )(\xi-kt )}{(k^2+(\xi-kt)^2)^{\frac 3 2}}\right\vert&= \left\vert \tfrac {(2 +(\frac\xi k -t )^2 )(\frac \xi k-t )}{(1+(\frac \xi k-t)^2)^{\frac 3 2}}\right\vert \\
    &\le \sup_s \left( \tfrac {(2+s^2)s}{(1+s^2)^{\frac 3 2 }} \right)\\
    &\le \sqrt{\tfrac 23 }.
\end{align*}

To estimate $L_1 $, we integrate in time and integrate by parts in space to
deduce that 
\begin{align*}
    & \quad \int d\tau \sum_k \int d\xi A^2 \tfrac {k (\xi-kt )}{k^2+(\xi-kt)^2}\tfrac {-1}{i\alpha k } \partial_t (p_1  \overline  p_2)\\
    &= \left[ \tfrac {-1}{i\alpha } \sum_k \int d\xi A^2 \tfrac {(\xi-kt )}{k^2+(\xi-kt)^2}p_1  \overline  p_2\right]^t_0 \\
    & \quad +\int d\tau \tfrac {1}{i\alpha  } \sum_k \int d\xi p_1  \overline  p_2 \partial_t ( A^2 \tfrac { (\xi-kt )}{k^2+(\xi-kt)^2}) \\
    &= \left[ \tfrac {-1}{i\alpha } \sum_k \int d\xi A^2 \tfrac {(\xi-kt )}{k^2+(\xi-kt)^2}p_1  \overline  p_2\right]^t_0 \\
    &\quad +\int d\tau \tfrac {2}{i\alpha  } \sum_k \int d\xi p_1  \overline  p_2 \tfrac {\dot M } M  A^2 \tfrac { (\xi-kt )}{k^2+(\xi-kt)^2} \\
    &\quad +c \mu \textbf{1}_{A=A^{N'}_\mu } \int d\tau \tfrac {2}{i\alpha  } \sum_k \int d\xi p_1  \overline  p_2   A^2 \tfrac { (\xi-kt )}{k^2+(\xi-kt)^2} \\
    &\quad +\int d\tau \tfrac {1}{i\alpha  } \sum_k \int d\xi p_1  \overline  p_2  A^2 \tfrac { k(k^2-(kt-\xi)^2)}{(k^2+(\xi-kt)^2)^2}.
\end{align*}
So we infer by Hölder's inequality that
\begin{align*}
& \quad \int d\tau \sum_k \int d\xi A^2 \tfrac {k (\xi-kt )}{k^2+(\xi-kt)^2}\tfrac {-1}{i\alpha k } \partial_t (p_1  \overline  p_2)\\
    &\le \tfrac 1 \alpha  (\Vert A p_1 (0)\Vert_{ L^2 } \Vert A p_2(0) \Vert_{L^2 }+ \Vert A p_1 (t)\Vert_{ L^2 } \Vert A p_2 (t)\Vert_{ L^2 }) \\
    &\quad + \mu \Vert A \p_x  p_1\Vert_{L^2  L^2 } \Vert A\sqrt{ -\tfrac {\dot M } M } p_2 \Vert_{L^2  L^2 }\\
    &\quad + \tfrac 1\alpha  \Vert A \sqrt{- \tfrac {\dot M } M } p_1 \Vert_{L^2L^2 } \Vert A \sqrt{ -\tfrac {\dot M } M } p_2 \Vert_{L^2  L^2 }
\end{align*}
and thus 
\begin{align*}
  & \quad \int  L d \tau -\int ONL d\tau \\
  &\le \tfrac 1  {2\alpha } (\Vert A p_1 (0)\Vert_{L^2 }^2+ \Vert A p_2(0) \Vert_{ L^2 }^2)\\
    &\quad + \tfrac 1  {2\alpha } (\Vert A p_1 \Vert_{ L^\infty L^2 }^2+ \Vert A p_2(t) \Vert_{ L^\infty L^2 }^2)\\
    &\quad + \tfrac {1 } \alpha (\mu  (1-c) \Vert \p_x A p\Vert_{L^2}^2+\mu (1-c) \Vert \p_y^t  A p\Vert_{L^2}^2+ \Vert \sqrt{- \tfrac {\dot M }M } A p\Vert_{L^2} ^2) .
\end{align*}
Using the dissipation estimates \eqref{eq:simpleenergy}, we therefore obtain 
\begin{align}
    \int  L d \tau  &\le \tfrac 1{2\alpha }(c_0+1)  \eps ^2 + \int ONL d\tau, \label{est:L}
\end{align}
where the $ONL$ part will be estimated at the beginning of the next subsection.

\subsection{Immediate nonlinear estimates for $A^N$ }\label{dnl}
In this subsection, we collect some estimates which can be obtained in a
straight forward approach using standard techniques (e.g. see \cite{bedrossian2016sobolev}). In particular, for these terms we are not constrained by the lack of vertical resistivity.
For most estimates we do not aim to establish optimal (mixing enhanced) bounds, since these bounds are in any case better than the ones involving horizontal resistivity and hence do not affect the over all stability threshold. In the following we write $A=A^N$.

\textbf{ONL estimate:} Using integration by parts in space and Hölder's
inequality, the nonlinear contribution in \eqref{eq:ONL} can be estimated by
\begin{align*}
    ONL&= \tfrac 2 {\alpha }\langle  A  \partial_y^t \Delta_t^{-1}  v_{\neq} ,A  (b\nabla_t b -v\nabla_t v )_{\neq}\rangle \\
    &\quad + \tfrac 2 {\alpha }\langle  A  \partial_y^t \Delta_t^{-1} b_{\neq} ,A   (b\nabla_t v-v\nabla_t b)_{\neq}\rangle \\
    &= \tfrac 2 {\alpha }\langle  A  \partial_y^t \Delta_t^{-1}  (\nabla^\perp_t \otimes v_{\neq}) ,A  (b\otimes  b -v\otimes v )_{\neq}\rangle \\
    &\quad + \tfrac 2 {\alpha }\langle  A  \partial_y^t \Delta_t^{-1} (\nabla^\perp_t \otimes b _{\neq}) ,A   (b\otimes v-v\otimes b)_{\neq}\rangle \\
    &\lesssim \tfrac 2 \alpha \Vert  A (v, b)_{\neq} \Vert^2_{L^2}\Vert A (v, b) \Vert_{L^2}.
\end{align*}
Recalling the bounds \eqref{eq:simpleenergy} and integrating in time we thus
obtain that 
\begin{align}
    \int ONL d\tau \lesssim \mu^{-1} \eps^3.\label{est:ONL}
\end{align}

\textbf{Estimates with an $x$-average in the second component:}
Let $a^1a^2a^3\in \{vvv,vbb,bbv,bvb\}$, then we need to estimate the trilinear products
\begin{align*}
    \langle A a^1_{\neq}  , A(a^2_{=} \nabla_t a^3_{\neq} ) \rangle &= \langle A a^1_{\neq} , A(a^2_{=,1} \partial_x a^3_{\neq})\rangle\\
    &\lesssim \Vert A a^1_{\neq}\Vert_{L^2}\Vert A  a^2_{=,1 } \Vert_{L^2}\Vert A \partial_x  a^3_{\neq}\Vert_{L^2}. 
\end{align*}
Integrating in time and again using the bound  \eqref{eq:simpleenergy} yields
a control by 
\begin{align}
    \int d \tau \langle Aa^1 , A(a^2_{=} \nabla_t a^3) \rangle &\lesssim \mu^{-1} \eps^3.\label{est:aver}
\end{align}
The influence of the underlying $x$-averaged velocity and magnetic field on the
average-less parts can thus be easily controlled by the dissipation, provided
$\epsilon\ll \mu$.
In the following we focus on terms involving $a^2_{\neq}$. 

\textbf{vvv estimate:}
We first discuss the velocity non-linearity and use the algebra property of
$H^N$ and the bounds on $A$ to estimate
\begin{align*}
    \langle A  v , A  v_{\neq}\nabla_t v\rangle&\le \Vert A v\Vert_{L^2}\Vert A v_{\neq}\Vert_{L^2}\Vert A\nabla_t v\Vert_{L^2}.
\end{align*}
Here, the contribution by $\|A \nabla_t v\|_{L^2}$ is square integrable in time
due to the dissipation \eqref{eq:simpleenergy}, while $\|A v_{\neq}\|_{L^2}$ is square
integrable in time by the inviscid damping estimates \eqref{damp}.
Integrating in time thus yields a bound by
\begin{align}
     \int d \tau \langle A  v , A  (v_{\neq}\nabla_t v)\rangle &\lesssim \mu^{-1} \eps^3 .\label{est:vvv}
\end{align}

\textbf{vbb estimate:}
For the contributions by the $vbb$ nonlinearity we intend to argue similarly,
but have to account for the lack of vertical magnetic dissipation (which we
compensate for by using the full fluid dissipation).
We may split the integral as 
\begin{align*}
    \langle A  v , A  (b_{\neq}\nabla_tb)\rangle&= \int A v_1 A(b_{1,\neq}\partial_x +b_{2,\neq}\partial_y^t)b_1\\
                                    &\quad +\int A v_2 A(b_{1,\neq}\partial_x +b_{2,\neq}\partial_y^t) b_2.
\end{align*}
For the second term we integrate by parts to obtain
\begin{align*}
    \int A v_1 A(b_{2,\neq}\partial_y^t b_1)&=- \int A  \partial_y^t v_1 A( b_{2,\neq} b_1)-\int A v_1 A(\partial_y^t  b_{2,\neq}b_1) .
\end{align*}
Furthermore, since $b$ is divergence-free, it holds that $\partial_y^t b_2 = -\p_x b_1 $ and hence 
\begin{align*}
    \langle A  v , A  (b_{\neq}\nabla_tb)\rangle
    &\le \Vert A v\Vert_{L^2}  \Vert A b_{\neq}\Vert_{L^2}  \Vert A \partial_x  b\Vert_{L^2} + \Vert \partial_y^t  v\Vert_{L^2}  \Vert A b_{\neq}\Vert_{L^2}  \Vert  A b_2\Vert_{L^2}.
\end{align*}
We may therefore estimate this term using the full fluid and horizontal magnetic dissipation \eqref{eq:simpleenergy} and integrating in time yields a bound by 
\begin{align}
\label{est:vbb}
    \int d \tau \langle A  v , A  (b_{\neq}\nabla_tb)\rangle &\lesssim \mu^{-1}\eps^3.
\end{align}
\textbf{bbv  estimate:} 
Finally, for the $bbv$ contribution, we may again use the full fluid dissipation and the algebra property of $A$ (and $H^N$) to obtain a bound
\begin{align*}
    \langle A  b  ,  A ( b_{\neq}\nabla_t v)\rangle&\lesssim \Vert Ab \Vert_{L^2}\Vert Ab_{\neq}  \Vert_{L^2}\Vert A \nabla_t v\Vert_{L^2}. 
\end{align*}
Integrating in time and using \eqref{eq:simpleenergy} we thus obtain a bound by 
\begin{align}
\label{est:bbv}
    \int d\tau \langle A  b  ,  A ( b_{\neq}\nabla_t v)\rangle&\lesssim \mu^{-1} \eps^3.
\end{align}

\subsection{High frequency $bvb$ term without $x$-average}\label{hfw}
Having established several straightforward estimates using the full fluid dissipation, in this and the following subsections we establish bounds for the high frequency (that is, $A^N$ terms as in \eqref{eq:hfassump}) terms involving $bvb$.
For simplicity, we write $A=A^N$ and aim to establish the estimate 
\begin{align*}
    \langle A  b  ,  A ( v_{\neq}\nabla_t b)\rangle\lesssim \mu^{-\frac 3 2 } \eps^3.
\end{align*}
We split the $bvb$ term according to (non)vanishing $x$-averages:
\begin{align*}
    \langle A  b  ,  A ( v_{\neq}\nabla_t b)\rangle&= \langle A  b_{\neq}  ,  A ( v_{\neq}\nabla_t b_{\neq})_{\neq}\rangle\\
    &\quad + \langle A  b_{\neq}  ,  A ( v_{\neq}\nabla_t b_{=})_{\neq}\rangle\\
    &\quad + \langle A  b_=  ,  A ( v_{\neq}\nabla_t b_{\neq})_=\rangle.
\end{align*}
Let us first consider the term without any $x$-averages, which can be written as 
\begin{align*}
    \langle A  b_{\neq}  ,  A ( v_{\neq}\nabla_t b_{\neq})\rangle&= \int A b_{1,\neq} A((v_{1,\neq}\partial_x +v_{2,\neq }\partial_y^t) b_{1,\neq})\\
    &\quad +\int A b_{2,\neq} A((v_{1,\neq}\partial_x +v_{2,\neq}\partial_y^t) b_{2,\neq}).
\end{align*}
We estimate the second contribution using the algebra property of $H^N$ and that $\partial_y^tb_2=-\partial_x b_1$, since $b$ is divergence-free:
\begin{align*}
    \int d \tau  \int &  A b_{2,\neq} A(v_{1,\neq}\partial_x +v_{2,\neq}\partial_y^t) b_{2,\neq}\\
    &\le \int d \tau \Vert A b_{2,\neq}\Vert_{L^2} (\Vert Av_{1,\neq}\Vert_{L^2}\Vert A \partial_x b_{2,\neq} \Vert_{L^2}  +\Vert Av_{2,\neq}\Vert_{L^2}\Vert A\partial_y^t b_{2,\neq}\Vert_{L^2} ) \\
&\le \int d \tau \Vert A b_{2,\neq}\Vert_{L^2} (\Vert Av_{1,\neq}\Vert_{L^2}\Vert \partial_x b_{2,\neq} \Vert_{L^2}  +\Vert Av_{2,\neq}\Vert_{L^2}\Vert A \partial_x b_{1,\neq}\Vert_{L^2} ).
\end{align*}
Employing Hölder's inequality this contribution can thus be estimated as 
\begin{align}
\begin{split}
\int d \tau  \int &  A b_{2,\neq} A((v_{1,\neq}\partial_x +v_{2,\neq}\partial_y^t) b_{2,\neq})\\
&\le\int d \tau  \Vert A b_{2,\neq}\Vert_{L^2} \Vert Av_{\neq}\Vert_{L^2}\Vert A \partial_x b_{\neq} \Vert_{L^2}\\
&\le \Vert A b_{2,\neq}\Vert_{L^2L^2} \Vert Av_{\neq}\Vert_{L^\infty L^2}\Vert A \partial_x b_{\neq} \Vert_{L^2L^2}\\
&\lesssim  \mu^{-1} \eps^3.\label{est:bvbn1}
\end{split}
\end{align}
It remains to control the contribution by $b_{1,\neq}$, which in view to the lack of vertical resistivity is the hardest term to control.
Since the velocity field $v$ is divergence-free, we observe that 
\begin{align*}
    \int A b_{1,\neq }(v_{\neq}\nabla_t  A b_{1,\neq })=0.
\end{align*}
Therefore, we obtain the following cancellations and introduce a splitting in Fourier space:
\begin{align*}
    \int A b_{1,\neq } A(v_{\neq} \nabla_t  b_{1,\neq })&= \int A b_{1,\neq } (A(v\nabla_t  b_{1,\neq })-(v_{\neq}\nabla_t  A b_{1,\neq }))\\
    &= \sum_{k,l,k-l \neq 0} \iint d(\xi,\eta ) A(k,\xi )b_1(k,\xi)   \tfrac {(A(k,\xi) -A(l,\eta))(\xi l -\eta k) }{\sqrt{(k-l)^2+(\xi-\eta-(k-l)t)^2}}\\
    &\qquad\qquad  p_1(k-l ,\xi-\eta)b_1(l,\eta) \\
    &=T+R+\calR. 
\end{align*}
Here, the Fourier regions
\begin{align*}
    \Omega_T&=\{\vert k-l,\xi-\eta \vert \le \tfrac 18 \vert l , \eta\vert \}, \\
    \Omega_R&=\{\vert l , \eta\vert  \le \tfrac 18 \vert k-l,\xi-\eta \vert\},\\
    \Omega_\calR &=\{\tfrac 18 \vert l , \eta\vert  \le  \vert k-l,\xi-\eta \vert\le 8\vert l , \eta\vert \},
\end{align*}
correspond to the the \emph{transport} ($T$) or low-high term, \emph{reaction} ($R$) or high-low term and the \emph{remainder} ($\calR $) or high-high term. In the following we omit the $\neq$ subscripts. 

\textbf{Transport term:} 
Since $ \vert k-l,\xi-\eta \vert \le \tfrac 18 \vert l , \eta\vert$ we obtain that $ \vert l,\eta \vert \approx  \vert k,\xi \vert$.
Without loss of generality we assume that $\xi \le \eta $, since we can use either of the following splittings 
\begin{align*}
    \xi l -k \eta &= (\xi-\eta)l -(k-l)\eta \\
    &= (\xi-\eta)k - \xi(k-l).
\end{align*}
Thus using the second equality we estimate
\begin{align*}
    T &\le \Vert \partial_y\Lambda^{-1}_t p_1 \Vert_{L^\infty } \Vert Ab_1 \Vert_{L^2} \Vert \partial_x A b_1\Vert_{L^2}\\
    &\quad + \sum_{k,l\neq 0} \iint d(\xi,\eta )  \textbf{1}_ {\Omega_T }     (\textbf{1}_{2\langle t\rangle ( k \vee l)\ge \xi}+\textbf{1}_{2\langle t\rangle ( k \vee l)\le \xi })   \\
    &\qquad\qquad \cdot A(k,\xi )b_1(k,\xi)   \tfrac {(A(k,\xi) -A(l,\eta))\xi (l - k) }{\sqrt{(k-l)^2+(\xi-\eta-(k-l)t)^2}}  p_1(k-l ,\xi-\eta)b_1(l,\eta),
\end{align*}
where we distinguished between $ 2 \langle t\rangle ( k \vee l)\ge \xi $  and $ 2 \langle t\rangle ( k \vee l)\le \xi $.

The first case is estimated by using the dissipation and \eqref{damp}:
\begin{align*}
    \sum_{k,l\neq 0} \iint d(\xi,\eta ) &\textbf{1}_{\Omega_T }   \textbf{1}_{\xi \le 2 ( k \vee l)\langle t\rangle  }  A(k,\xi )b_1(k,\xi)\\
    &\cdot    \tfrac {(A(k,\xi) -A(l,\eta))\xi (l -k) }{\sqrt{(k-l)^2+(\xi-\eta-(k-l)t)^2}} p_1(k-l ,\xi-\eta)b_1(l,\eta) \\
    &\lesssim \langle t\rangle   \Vert A b_1 \Vert_{L^2}  \Vert \Lambda^{-1}_t  \partial_x p_1 \Vert_{L^\infty } \Vert \partial_x A b_1\Vert_{L^2 } \\
    &\lesssim \Vert A b_1 \Vert_{L^2}  \Vert \Lambda \partial_x p_1 \Vert_{L^\infty } \Vert \partial_x A b_1\Vert_{L^2 } \\
    &\lesssim \Vert A b_1 \Vert_{L^2}  \Vert A p_1 \Vert_{L^2} \Vert \partial_x A b_1\Vert_{L^2 }. 
\end{align*}
For the second case,  $ 2 \langle t\rangle ( k \vee l)\le \xi $, we need to estimate
\begin{align*}
    (A^N(k,\xi) - A^N(l,\eta)) &= (M(k,\xi)\vert k,\xi\vert^N -  M(l,\eta)\vert l,\eta \vert^N )\\
    &= M(k,\xi)(\vert k,\xi\vert^N - \vert l,\eta \vert^N)\\
    &\quad +  M(l,\eta)(\tfrac {M(k,\xi)}{  M(l,\eta)}-1)\vert l,\eta \vert^N.
\end{align*}
By the mean value theorem, we obtain 
\begin{align*}
    \vert k,\xi\vert^N - \vert l,\eta \vert^N &\le N  \vert k-\theta l ,\xi- \theta \eta \vert^{N -1}\vert k-l , \xi- \eta \vert \\
    &\lesssim\vert k-l , \xi- \eta \vert ( \vert l,\eta \vert^{N-1} + \vert k-l , \xi- \eta \vert^{N-1} )\\
    &\lesssim\vert k-l , \xi- \eta \vert \vert l,\eta \vert^{N-1} . 
\end{align*}
For the differences in $M$ we use that for $a,b>0$ it holds that  $\vert e^{a-b}-1\vert \le e^{a+b}-1$ and hence
\begin{align*}
    \vert \tfrac {M_1(k,\xi)}{  M_1(l,\eta)}-1\vert &= \vert \exp\left( \int_0^t \tfrac {\vert l \vert} {l^2 + (\eta -ls)^2 }-\tfrac {\vert k \vert} {k^2 + (\xi -ks)^2 }  ds  \right) -1 \vert \\
    &\le \vert \exp\left( \int_0^t \tfrac {\vert l \vert} {l^2 + (\eta -ls)^2 }+\tfrac {\vert k \vert} {k^2 + (\xi -ks)^2 }  ds  \right) -1 \vert. 
\end{align*}
Thus for $\eta \ge \xi\ge 2 t(k\vee l ) $ by integrating we obtain that
\begin{align*}
    \vert \tfrac {M_1(k,\xi)}{  M_1(l,\eta)}-1\vert &\le \exp\left( \tfrac 1 {\vert l \vert }  \int_0^t \tfrac {1} {1 + (\frac \eta l -s)^2 }ds + \tfrac 1 {\vert k \vert }\int_0^t \tfrac {1} {1 + (\frac\xi k  -s)^2 }  ds \right) -1\\
    &\le \exp(\tfrac {1} {\eta } +\tfrac {1}\xi  ) -1 \\
    &\lesssim \tfrac {1} \eta+\tfrac 1\xi .
\end{align*}
Therefore, we deduce that 
\begin{align*}
    \sum_{k,l,k-l \neq 0} \iint d(\xi,\eta ) &\textbf{1}_ {\Omega_T }\textbf{1}_{\xi \ge 2 (k\vee l) t } 
    A(k,\xi )b_1(k,\xi)   \tfrac {(A(k,\xi) -A(l,\eta))\xi(l - k) }{\sqrt{(k-l)^2+(\xi-\eta-(k-l)t)^2}} \\
    &\qquad\qquad  p_1(k-l ,\xi-\eta)b_1(l,\eta) \\
    &\lesssim \Vert A b_1 \Vert_{L^2} \Vert \Lambda_t^{-1} \partial_x p_1 \Vert_{L^\infty } \Vert   A b_1\Vert_{L^2 }\\
    &\lesssim  \langle t\rangle^{-1} \Vert A b_1 \Vert_{L^2}  \Vert \Lambda \partial_x p_1 \Vert_{L^\infty } \Vert   A b_1\Vert_{L^2 }\\
    &\lesssim  \langle t\rangle^{-1} \Vert A b_1 \Vert_{L^2}  \Vert A p_1 \Vert_{L^2 } \Vert   A b_1\Vert_{L^2},
\end{align*}
where we used the estimate \eqref{damp}. Combining all estimates, we have derived the following estimate of the transport term: 
\begin{align}
\begin{split}
    \int Td\tau &\lesssim  \Vert A b_1 \Vert_{L^\infty L^2} \Vert A p_1 \Vert_{L^\infty L^2}\Vert A  b_1\Vert_{L^2L^2}\\
    &\lesssim   \mu^{-\frac 1 2  } \eps^3.\label{est:bvbnT}
\end{split}
\end{align}

\textbf{Reaction term:} 
Since $\vert l , \eta\vert  \le \tfrac 18 \vert k-l,\xi-\eta \vert$ we obtain $ \vert k-l,\xi-\eta \vert \approx  \vert k,\xi \vert$. With the identity 
\begin{align*}
    \xi l -k \eta &= l(\xi-\eta -(k-l) t) - (k-l) (\eta -lt) 
\end{align*}
and $A(k,\xi) -A(l,\eta)\lesssim A(k-l,\xi-\eta)$ we infer 
\begin{align*}
    R&= \sum_{k,l,k-l \neq 0} \iint d(\xi,\eta ) \textbf{1}_{\Omega_R }A(k,\xi )b_1(k,\xi)   \tfrac {(A(k,\xi) -A(l,\eta)) ( l(\xi-\eta -(k-l) t) - (k-l) (\eta -lt) ) }{\sqrt{(k-l)^2+(\xi-\eta-(k-l)t)^2}} \\
    &\qquad\qquad \cdot  p_1(k-l ,\xi-\eta)b_1(l,\eta) \\
    &\le \Vert A b_1 \Vert_{L^2} \Vert A \partial_y^t \Lambda_t^{-1} p_1  \Vert_{L^2 } \Vert \partial_x b_1 \Vert_{L^\infty }  \\
    &\quad +\Vert A b_1 \Vert_{L^2} \Vert A  \Lambda_t^{-1} p_1  \Vert_{L^2 }\Vert \partial_y^t\partial_x^2  b_1\Vert_{L^\infty } \\
    &\quad +\Vert \partial_x A b_1 \Vert_{L^2} \Vert A  \Lambda_t^{-1} p_1  \Vert_{L^2 }\Vert \partial_y^t \partial_x b_1\Vert_{L^\infty } .
\end{align*}
We split $\partial_y^t = \partial_y-t\p_x$ and use the definition of the low-frequency multiplier $A^{N'}_\mu $ to estimate
\begin{align*}
    \Vert \langle \partial_x \rangle^2 \partial_y^t  b_1\Vert_{L^\infty }&\le\Vert \langle \partial_x \rangle^2 \partial_y  b_1\Vert_{L^\infty }+\Vert \langle \partial_x \rangle^2 t\partial_x  b_1\Vert_{L^\infty }\\
    &\le t \Vert \Lambda^{N'}  b_1\Vert_{L^2} \\
    &\lesssim te^{-c\mu t }\Vert   A^{N'}_\mu   b_1\Vert_{L^2 }\\
    &\lesssim \mu^{-1} \Vert  A^{N'}_\mu  b_1\Vert_{L^2 }.
\end{align*}
Therefore, integrating in time yields the estimate
\begin{align}
\label{est:bvbnR}  
\begin{split}
    \int Rd \tau &\lesssim \Vert A b_1 \Vert_{L^2 L^2} \left(\Vert A \partial_y^t \Lambda_t^{-1} p_1  \Vert_{L^2L^2 } \Vert A b_1 \Vert_{L^\infty  L^2 }\right)  \\
    &\quad +\mu^{-1} \Vert A \partial_x   b_1 \Vert_{L^2 L^2}\Vert A  \Lambda_t^{-1} p_1  \Vert_{L^2L^2 } \Vert  A^{N'}_\mu b_1\Vert_{L^\infty L^2 }\\
    &\lesssim \eps^3   \mu^{-\frac 32 }.  
\end{split}
\end{align}

\textbf{$\calR $ term:}
We consider the Fourier region where $ \tfrac 18 \vert l , \eta\vert  \le  \vert k-l,\xi-\eta \vert\le 8\vert l , \eta\vert $. Thus, we have the bounds $\vert k, \xi \vert \lesssim \vert l, \eta \vert$ and $A( k , \xi )\lesssim A(l,\eta )\approx A(k-l, \xi-\eta) $. 
Furthermore, we note that
\begin{align*}
    \xi l -\eta k&\le \vert l, \eta  \vert^2,
\end{align*}
and thus estimate the remainder terms as
\begin{align*}
    \calR &= \sum_{k,l,k-l \neq 0} \iint d(\xi,\eta )1_{\Omega_\calR  } A(k,\xi )b_1(k,\xi)  \\
    &\qquad\qquad \tfrac {(A(k,\xi) -A(l,\eta))(\xi l -\eta k) }{\sqrt{(k-l)^2+(\xi-\eta-(k-l)t)^2}}  p_1(k-l ,\xi-\eta)b_1(l,\eta) \\
    &\lesssim  \Vert A b_1 \Vert_{L^2} \Vert A \Lambda^{-1}_t p_1 \Vert_{L^2} \Vert \Lambda^2 b_1 \Vert_{L^\infty }\\
    &\lesssim \Vert A b_1 \Vert_{L^2} \Vert A \Lambda^{-1}_t p_1 \Vert_{L^2} \Vert A b_1 \Vert_{L^2 }.
\end{align*}
Hence after integrating in time, we deduce that
\begin{align}
\label{est:bvbncalR}  
    \int\calR &\lesssim \ \Vert A b_1 \Vert_{L^2 L^2} \Vert \sqrt{-\tfrac {\dot M} M} A p_1 \Vert_{L^2L^2} \Vert A b_1 \Vert_{L^\infty  L^2}\lesssim\mu^{-\frac 1 2 } \eps^3.   
\end{align}

Combining the estimates \eqref{est:bvbn1}, \eqref{est:bvbnT}, \eqref{est:bvbnR} and \eqref{est:bvbncalR}, we finally conclude that
\begin{align}
\label{est:bvbn}
    \langle A  b_{\neq}  ,  A ( v_{\neq}\nabla_t b_{\neq})_{\neq}\rangle\lesssim \mu^{-\frac 3 2 } \eps^3.
\end{align}

\subsection{High frequency estimates for $bvb$ terms with $x$-averages} \label{hfa}
In this subsection we aim to estimate the remaining terms in the $bvb$ integrals, which involve $x$-averages. 
We consider the two terms 
\begin{align*}
    \langle A  b_{\neq}  &,  A ( v_{\neq}\nabla_t b_{=})_{\neq}\rangle+ \langle A  b_=  ,  A ( v_{\neq}\nabla_t b_{\neq})_=\rangle\\
    &= \langle A  b_{1,\neq}  ,  A ( v_{\neq}\nabla_t b_{1,=})_{\neq}\rangle+ \langle A  b_{1,=}  ,  A ( v_{\neq}\nabla_t b_{1,\neq})_=\rangle,
\end{align*}
where we used that $b_{2,=}=0$, since $b$ is divergence-free. Using integration by parts and the fact that $v$ is divergence-free, we obtain that
\begin{align*}
    \langle A  b_{1,\neq}  ,   v_{\neq}\nabla_t A  b_{1,=}\rangle& + \langle A  b_{1,=} ,   v_{\neq}\nabla_t A b_{1,\neq}\rangle\\
    &= \langle v_{\neq}  ,   \nabla_t (A  b_{1,=}A b_{1,\neq})\rangle=0, 
\end{align*}
and thus 
\begin{align*}
    \langle A  b_{1,\neq}  &,  A ( v_{\neq}\nabla_t b_{1,=})\rangle+ \langle A  b_{1,=}  ,  A ( v_{\neq}\nabla_t b_{1,\neq})\rangle\\
    &=\langle A  b_{1,\neq}  ,  A ( v_{\neq}\nabla_t b_{1,=}) -  v_{\neq}\nabla_t A  b_{1,=}\rangle + \langle A  b_{1,=}  ,  A ( v_{\neq}\nabla_t b_{1,\neq})-  v_{\neq}\nabla_t A b_{1,\neq}\rangle\\
    &= \sum_{k\neq 0} \iint d(\xi,\eta) A(k,\xi) b_1 ( k ,\xi) \tfrac {(A(k,\xi)- A(0,\eta)) (-k\eta)}{\sqrt{k^2+(\xi-\eta -kt)^2} }p_1(k,\xi-\eta ) b_1(0, \eta)\\
    &\quad + \sum_{k\neq 0} \iint d(\xi,\eta) A(0,\xi) b_1 ( 0 ,\xi) \tfrac {(A(0,\xi)- A(k,\eta)) (-k\xi) }{\sqrt{k^2+(\xi-\eta -kt)^2} }p_1(k,\xi-\eta ) b_1(-k, \eta).
\end{align*}
Again we split this integrals into the transport $T$, reaction $R$ and remainder terms $\calR$ with the same definition of sets in Fourier space:
\begin{align*}
    \Omega_T&=\{\vert \xi-\eta \vert \le \tfrac 18 \vert  \eta\vert \}, \\
    \Omega_R&=\{\vert  \eta\vert  \le \tfrac 18 \vert \xi-\eta \vert\},\\
    \Omega_\calR &=\{\tfrac 18 \vert  \eta\vert  \le  \vert \xi-\eta \vert\le 8\vert  \eta\vert \}.
\end{align*}

\textbf{Transport term:}
Since $ \vert \xi-\eta \vert \le \tfrac 18 \vert  \eta\vert$ we obtain that $ \vert \eta \vert \approx  \vert \xi \vert$. 

In our estimates, we distinguish the cases $\xi \vee \eta \le 2 k\langle t \rangle $ and $\xi \vee \eta \ge 2  k \langle t \rangle$.
In the first case,  $\xi \vee \eta \le 2 k\langle t \rangle $ we obtain a bound by 
\begin{align*}
  \sum_{k\neq 0}  & \iint d(\xi,\eta) \textbf{1}_{\Omega_T}\textbf{1}_ { \xi \vee \eta \le k\langle t \rangle }  A(k,\xi) b_1 ( k ,\xi) \tfrac {(A(k,\xi)- A(0,\eta)) k\eta }{\sqrt{k^2+(\xi-\eta -kt)^2} }p_1(k,\xi-\eta ) b_1(0, \eta)\\
    + \sum_{k\neq 0} & \iint d(\xi,\eta)\textbf{1}_{\Omega_T} \textbf{1}_ {\xi \vee \eta  \le k\langle t \rangle }  A(0,\xi) b_1 ( 0 ,\xi) \tfrac {(A(0,\xi)- A(k,\eta)) k\xi }{\sqrt{k^2+(\xi-\eta -kt)^2} }p_1(k,\xi-\eta ) b_1(-k, \eta)\\
    &\le t  \Vert A b_= \Vert_{L^2 } \Vert \partial_x^2\Lambda_t^{-1} p_{1,\neq} \Vert_{L^\infty  } \Vert A b_{1,\neq}\Vert_{L^2}\\
    &\lesssim  \Vert A b_= \Vert_{L^2 } \Vert A p_{1,\neq} \Vert_{L^2} \Vert Ab_{1,\neq}\Vert_{L^2}.
\end{align*}
In the case $\xi \vee \eta \ge 2 k\langle t \rangle $, we instead estimate
\begin{align*}
     A(k,\xi)- A(0,\eta)&\le M(k,\xi) (\xi^2+k^2)^{\frac N 2 } -\eta^N \\
     &= (M(k,\xi)-1) (\xi^2+k^2)^{\frac N 2 } +((\xi^2+k^2)^{\frac N 2 }- \eta^N). 
\end{align*}
Since $\xi \ge 2 k\langle t \rangle$, in the first summand we may bound
\begin{align*}
    M(k,\xi)-1&= \exp\left( - \int_0^t \tfrac {\vert k \vert} {k^2 + (\xi -ks)^2 }  ds\right)-1 \\
    &\lesssim \tfrac 1 \xi \lesssim \tfrac 1 \eta .
\end{align*}
By the mean value theorem we further infer
\begin{align*}
    (\xi^2+k^2)^{\frac N 2 }- \eta^N&\le ((\xi-\theta \eta )^2+k^2)^{\frac {N-1} 2 }\vert k,\xi - \eta \vert  \lesssim \vert k,\xi - \eta \vert (\xi^2+k^2)^{\frac {N-1}  2}.
\end{align*}
Thus, using that $k\le\xi \lesssim \eta$,  we deduce that
\begin{align*}
    A(k,\xi)- A(0,\eta)\lesssim \vert k,\xi - \eta \vert  \eta^{N -1},\\
    A(k,\eta)- A(0,\xi)\lesssim \vert k,\xi - \eta \vert \eta^{N -1},
\end{align*}
where the proof for $A(k,\eta)-A(0,\xi)$ is analogous. Finally, we obtain
\begin{align*}
    \langle A  b_{\neq}  &,\textbf{1}_{\Omega_T} \textbf{1}_ { \eta \ge kt}  A ( v_{\neq}\nabla_t b_{=})\rangle+ \langle A  b_=  ,\textbf{1}_{\Omega_T} 1_ { \eta \ge kt}  A ( v_{\neq}\nabla_t b_{\neq})\rangle\\
    &\lesssim  \sum_{k\neq 0} \iint d(\xi,\eta) A(k,\xi) b_1 ( k ,\xi) \tfrac {\vert k,\xi - \eta \vert  \eta^{N -1} }{\sqrt{k^2+(\xi-\eta -kt)^2} }p_1(k,\xi-\eta ) b_1(0, \eta)\\
    &\quad + \sum_{k\neq 0} \iint d(\xi,\eta) A(0,\xi) b_1 ( 0 ,\xi) \tfrac {\vert k,\xi - \eta \vert  \eta^{N -1} }{\sqrt{k^2+(\xi-\eta -kt)^2} }p_1(k,\xi-\eta ) b_1(k, \eta)\\
    &\lesssim \Vert A b_= \Vert_{L^\infty} \Vert  A   \Lambda_t^{-1}  p_{1,\neq} \Vert_{L^2 } \Vert A b_{1,\neq}\Vert_{L^2}\\
    &\lesssim \Vert A b_= \Vert_{L^\infty} \Vert   A \Lambda_t^{-1}  p_{1,\neq} \Vert_{L^2 } \Vert A b_{1,\neq}\Vert_{L^2},
\end{align*}
and integrating in time yields the desired bound:
\begin{align}
\begin{split}
    \int \langle A  b_{\neq}  &, \textbf{1}_{\Omega_T}  A ( v_{\neq}\nabla_t b_{=})\rangle+ \langle A  b_=  , \textbf{1}_{\Omega_T}    A ( v_{\neq}\nabla_t b_{\neq})\rangle d\tau \\
    &\lesssim \mu^{-1  } \eps^3 .\label{est:bvbaT}
\end{split}
\end{align}

\textbf{Reaction term:}
Since $ \vert \eta \vert \le \tfrac 18 \vert \xi - \eta\vert$ we obtain $ \vert \xi-  \eta \vert \approx  \vert \xi \vert$ and thus 
\begin{align*}
    R&=\langle  A  b_{\neq}  , \textbf{1}_{\Omega_R}  A( ( v_{\neq}\nabla_t b_{=})- v_{\neq}\nabla_t A b_{=})\rangle+ \langle  A  b_=  , \textbf{1}_{\Omega_R} (A ( v_{\neq}\nabla_t b_{\neq})- v_{\neq}\nabla_t Ab_{\neq})_=\rangle\\
    &\le  \sum_{k\neq 0} \iint d(\xi,\eta)\textbf{1}_{\Omega_R}   A(k,\xi) b_1 ( k ,\xi) \tfrac {(A(k,\xi)- A(0,\eta)) k\eta }{\sqrt{k^2+(\xi-\eta -kt)^2} }p_1(k,\xi-\eta ) b_1(0, \eta)\\
    &\quad + \sum_{k\neq 0} \iint d(\xi,\eta)\textbf{1}_{\Omega_R}   A(0,\xi) b_1 ( 0 ,\xi) \tfrac {(A(0,\xi)- A(-k,\eta)) k\xi}{\sqrt{k^2+(\xi-\eta -kt)^2} }p_1(k,\xi-\eta ) b_1(-k, \eta)\\
    &\lesssim \Vert A b_{1,\neq}\Vert_{L^2} \Vert A \partial_x \Lambda_t^{-1} p_{1,\neq} \Vert_{L^2} \Vert \partial_y b_{1,=}\Vert_{L^\infty } \\
    &\quad + \Vert A b_{1,=}\Vert_{L^2} \Vert A \partial_y \partial_x^{-1}  \Lambda_t^{-1} p_{1,\neq} \Vert_{L^2} \Vert \partial_x^2 b_{1,\neq}\Vert_{L^\infty }.
\end{align*}
Expressing $\p_y =\p_y^t+t\p_x $ in terms of the time-dependent derivatives, at this point we require the splitting into high and low frequency estimates. More precisely, using the additional time decay of the low-frequency part,  we estimate 
\begin{align*}
    \Vert A \partial_y \partial_x^{-1}  \Lambda_t^{-1} p_{1,\neq} \Vert_{L^2}&\le \Vert A \partial_y^t \partial_x^{-1}  \Lambda_t^{-1} p_{1,\neq} \Vert_{L^2}+t \Vert A \Lambda_t^{-1} p_{1,\neq} \Vert_{L^2}\\
    &\lesssim \Vert A  p_{1,\neq} \Vert_{L^2}+t \Vert A\Lambda^{-1}_t  p_{1,\neq} \Vert_{L^2}
\end{align*}
and using the definition of $A^{N'}_\mu $ we can absorb the growth of the factor $t$ at the cost of a power of $\mu$: 
\begin{align*}
    \Vert \partial_x^2 b_{1,\neq}\Vert_{L^\infty }&\le \Vert \Lambda^{N'} b_{1,\neq}\Vert_{L^2 }\\
    &\lesssim  e^{-c\mu t } \Vert A^{N'}_{\mu } b_{1,\neq}\Vert_{L^2 }\\
    &\lesssim \mu^{-1} \langle t \rangle^{-1}  \Vert A^{N'}_{\mu } b_{1,\neq}\Vert_{L^2 }.
\end{align*}
Thus we obtain 
\begin{align*}
    R&\lesssim \Vert A^Np_{1,\neq}\Vert_{L^2}  \Vert A^N b_{1,=}\Vert_{L^2 }    \Vert A^N b_{1,\neq}\Vert_{L^2} \\
    &\quad + \mu^{-1}  \Vert A^N b_{1,=}\Vert_{L^2 }    \Vert A \Lambda_t^{-1}  p_{1,\neq} \Vert_{L^2} \Vert A^{N'}_{\mu } b_{1,\neq}\Vert_{L^2 } .
\end{align*}
Integrating in time then yields the estimate
\begin{align}
    \int R d\tau &\lesssim \mu^{-\frac 3 2 }  \eps^3.\label{est:bvbaR}
\end{align}

\textbf{$\calR$  term:}
The remainder term $\calR$ can be estimated by the same argument as in the case without $x$-averages in Subsection \ref{hfw}.

Combining the estimates \eqref{est:bvbaT}, \eqref{est:bvbaR} and \eqref{est:bvbn}, we conclude that the $bvb$ term can be controlled as
\begin{align}
    \langle A  b  ,  A ( v_{\neq}\nabla_t b)\rangle\lesssim \mu^{-\frac 3 2 } \eps^3 .\label{est:bvb} 
\end{align}

\subsection{Low frequency estimates } \label{lf} 
In this subsection we establish the estimates on the low frequency errors. For simplicity of presentation we present the proof of these estimates for the $bvb$ nonlinearity. The estimates with an $x$-average in the second component are analogous to the ones in Subsection \ref{dnl}. The arguments for the $vvv$, $vbb$, $bbv$ or $ONL$ trilinear terms are also analogous.

We aim to establish the bound
\begin{align*}
    \langle A^{N'}_\mu   b  ,  A^{N'}_\mu  ( v_{\neq}\nabla_t b)\rangle&\lesssim \mu^{-\frac 12 }\eps^3,
\end{align*}
and, as in the previous section, separately discuss the transport, reaction and remainder term.

For the transport term, we note that  
\begin{align*}
     v_{\neq}\nabla_t &= \nabla_t^\perp \Lambda^{-1}_t p_1 \nabla_t \\
     &= \nabla^\perp \Lambda^{-1}_t p_1 \nabla.
\end{align*}
Hence, we may rewrite
\begin{align*}
    \langle A^{N'}_\mu   b  ,  A^{N'}_\mu  ( v_{\neq}\nabla_t b)\rangle&= \langle A^{N'}_\mu   b  ,  A^{N'}_\mu  ( \nabla^\perp \Lambda^{-1}_t p_{1,\neq}\nabla b)\rangle. 
\end{align*}
In a first step, we estimate the $b_{\neq}$ term by using the algebra property of $A^{N'}$:
\begin{align*}
    \langle A^{N'}_\mu   b  &,  A^{N'}_\mu  ( \nabla^\perp \Lambda^{-1}_t p_{1,\neq}\nabla b_{\neq})\rangle\\
    &\le \Vert A^{N'}_\mu   b \Vert_{L^2} e^{c\mu_x t} \big ( \Vert A^{N'}   \nabla^\perp \Lambda_t^{-1} p_{1,\neq}\Vert_{L^2} \Vert \nabla b_{\neq}\Vert_{L^\infty }+\\& \qquad \qquad     \Vert   \nabla^\perp \Lambda_t^{-1} p_{1,\neq}\Vert_{L^\infty } \Vert A^{N'} \nabla b_{\neq}\Vert_{L^2 }\big ) \\
    &\le \Vert A^{N'}_\mu   b\Vert_{L^2}\big ( \Vert A^{N}   \Lambda_t^{-1} p_{1,\neq}\Vert_{L^2} \Vert A^{N'}_\mu  b_{\neq}\Vert_{L^2 }+  \Vert   A_\mu^{N'} \Lambda_t^{-1} p_{1,\neq}\Vert_{L^2} \Vert A^{N} b_{\neq}\Vert_{L^2}\big). 
\end{align*}
Integrating in time then yields the estimate
\begin{align}
   \int d \tau  \langle A^{N'}_\mu   b ,  A^{N'}_\mu  ( v_{\neq}\nabla_t b_{\neq} )\rangle&\lesssim \mu^{-\frac 1 2 } \eps^3.\label{est:low1}
\end{align}
Furthermore, we estimate the $b_=$ term by partial integration and the algebra property of $A^{N' } $ 
\begin{align*}
   &\quad  \langle A^{N'}_\mu   b  ,  A^{N'}_\mu  ( \nabla^\perp \Lambda^{-1}_t p_{1,\neq}\nabla b_=)\rangle\\
   & =-\langle A^{N'}_\mu   b_{1,\neq} ,  A^{N'}_\mu  ( \p_x \Lambda^{-1}_t p_{1,\neq}\p_yb_{1,=} )\rangle\\
    &=\langle \p_x A^{N'}_\mu   b_{1,\neq} ,  A^{N'}_\mu  (  \Lambda^{-1}_t p_{1,\neq}\p_yb_{1,=} )\rangle\\
    &\le \Vert  \p_x A^{N'}_\mu   b_{1,\neq} \Vert_{L^2} e^{c\mu t } \big( \Vert  A^{N'}  \Lambda^{-1}_t p_{1,\neq}\Vert_{L^2}\Vert\p_yb_{1,=} \Vert_{L^\infty }\\
    &\qquad \qquad + \Vert   \Lambda^{-1}_t p_{1,\neq}\Vert_{L^\infty } \Vert  \p_y^{N'+1} b_{1,=} \Vert_{L^2}\big ) \\
     &\lesssim  \Vert\p_x  A^{N'}_\mu   b_{1,\neq} \Vert_{L^2}\big( \Vert  A^{N'}_\mu    \Lambda^{-1}_t p_{1, \neq}\Vert_{L^2}\Vert A^{N'}b_{1,=} \Vert_{L^2}\\
     &\qquad \qquad + \Vert   A^{N'}_\mu  \Lambda^{-1}_t p_{1, \neq}\Vert_{L^2} \Vert A^{N}b_{1,=} \Vert_{L^2}\big ).
\end{align*}
Integrating in time then yields that
\begin{align}
    \int d \tau \langle A^{N'}_\mu   b_{\neq} ,  A^{N'}_\mu  ( v_{\neq}\nabla_t b_{=} )\rangle &\lesssim \mu^{-\frac 12 }\eps^3 .\label{est:low2}
\end{align}

This concludes our proof of Proposition \ref{prop:errors} and hence of Theorem \ref{thm:anisoThres}.
More precisely, the claimed estimates for both $A^N$ and $A^{N'}_{\mu}$ are obtained by combining the respective linear estimate \eqref{est:L}, the high frequency nonlinear estimates \eqref{est:ONL}, \eqref{est:aver}, \eqref{est:vvv}, \eqref{est:vbb}, \eqref{est:bbv}, \eqref{est:bvb}, and  the low frequency estimates given in \eqref{est:low1} and \eqref{est:low2}. 

We emphasize that the stability threshold of $\tfrac 3 2 $ is determined by the estimates for the action of the $v\cdot \nabla_t b$ nonlinearity in the estimate \eqref{est:bvb} and, in particular, by the estimates of the reaction terms \eqref{est:bvbnR} and \eqref{est:bvbaR}. 
These estimates are expected to be optimal and together with the linear estimates of Section \ref{linstab} highlight the effects of the lack of vertical resistivity.

The partial dissipation case considered in this article
\begin{align*}
    \kappa_y=0, \ \nu_x=\nu_y=\kappa_x>0,
\end{align*}
shows the large impact of (partial) magnetic resistivity on the behavior of the MHD equations and the (de)stabilizing role of the magnetic field.
As mentioned following Theorem \ref{thm:anisoThres}, more generally our methods of proof extend to the case where $\kappa_x$ is bounded below in terms of $\nu$:
\begin{align*}
  \nu_y^{1/3}\geq \kappa_x \geq \frac{1}{2\alpha} \nu_y.  
\end{align*}
The complementary regime, where $\kappa_x$ tends to zero quicker than $\nu_{y}$ remains an interesting topic for future work. The limiting case, $\kappa_x=0$, and the associated instability is discussed in the following section.

\section{Instability of the non-resistive MHD system}\label{instab}

As a complementary result, in this section we consider the non-resistive MHD
equations and establish the instability estimates of Proposition \ref{prop:instability}.
\subsection{Linear instability}
We begin by studying the linearized MHD equations with isotropic viscosity and
vanishing resistivity:
\begin{align}
  \label{eq:liso}
\begin{split}
     \partial_t p_1 - \partial_x \partial_x^t \Delta^{-1}_t p_1- \alpha \partial_x p_2 &= \nu \Delta_t p_1, \\
      \partial_t p_2 +\partial_x \partial_x^t \Delta^{-1}_t p_2 - \alpha \partial_x p_1 &= 0.
\end{split}
\end{align}

\begin{lemma}[Quantitative linear instability of the non-resistive MHD equations]\label{lemma:linIns}
    For the linearized  equations \eqref{eq:liso} there exists initial data $p_{in}$ such that 
    \begin{align}
    \begin{split}
        \Vert p(t)\Vert_{H^N } &\ge t\tfrac {\nu }{8\alpha^2 }\Vert p_{in}\Vert_{H^N},\\
        \Vert p(t)\Vert_{H^{N-1} } &\ge t\tfrac {\nu^2 }{32\alpha^4 }\Vert p_{in}\Vert_{H^N }.\label{eq:isolow}
    \end{split}
    \end{align}
    Furthermore, for all solutions such that at time $\tau$ it holds $p(\tau)\in H^N$, then we obtain 
    \begin{align}
      \label{eq:isoup}
        \Vert p\Vert_{H^N}\le  {\langle t\rangle^2 }\Vert p(\tau)\Vert_{H^N}. 
    \end{align}
\end{lemma}

\begin{proof}[Proof of Lemma \ref{lemma:linIns}]
After a Fourier transform \eqref{eq:liso} yields 
\begin{align}
\begin{split}
     \partial_t p_1(k)  &= -\tfrac {t-\frac \xi k } {1+(t-\frac \xi k )^2} p_1(k) + \alpha  k p_2 (k) -\nu ( k^2 +  (\xi-kt)^2) p_1(k),  \\
      \partial_t p_2 (k)  &=\tfrac {t-\frac \xi k } {1+(t-\frac \xi k )^2} p_1(k)  - \alpha k p_1 (k).\label{eq_p_FT}
\end{split}
\end{align}

We assume that $p_1(0,k,\xi)=0$ and consider variables $k=-1$ and $\xi\ge 2\tfrac {\alpha^2}\nu  $
\begin{align*}
    p_1&= -\alpha  \int^t_0 d\tau \sqrt{\tfrac {1+(\tau_1+\xi )^2}{1+(t+\xi )^2}}\exp(-\nu (t-\tau+\tfrac 1 3 ((t+\xi  )^3-(\tau_1+\xi  )^3)))p_2(\tau_1) 
\end{align*}
thus we can estimate $p_2$ by
\begin{align*}
    &\quad p_2-\sqrt {\tfrac {1+(t+\xi )^2 } {1+\xi^2 }}p_{2,in}(k)\\
    &= -\alpha k \int_0^t d\tau_2\sqrt{\tfrac {1+(t+\xi )^2 } {1+(\tau_2 +\xi )^2 }} p_1(\tau_2,-1) \\
    &= -\alpha^2  \int_0^t d\tau_2 \int^{\tau_1}_0 d\tau_1 \tfrac {\sqrt{1+(t+\xi  )^2} \sqrt{1+(\tau_1+\xi  )^2}}{1+(\tau_2+\xi  )^2} p_2(\tau_1) \\
    &\qquad \cdot \exp(-\nu (\tau_2-\tau_1+\tfrac 1 3 ((\tau_2+\xi  )^3-(\tau_1+\xi )^3)))\\
    &\le {\alpha^2}\vert p_2\vert_\infty \int_0^t d\tau_1 \int^{t}_{\tau_2 } d\tau_2  \exp(-\nu (\tau_2-\tau_1+\tfrac 1 3 ((\tau_2+\xi  )^3-(\tau_1+\xi )^3))).
\end{align*}

We estimate the last integral by
\begin{align*}
    \int_0^t d\tau_1& \int^{t}_{\tau_2 } d\tau_2  \exp(-\nu (\tau_2-\tau_1+\tfrac 1 3 ((\tau_2+\xi  )^3-(\tau_1+\xi )^3)))\\
    &= \int_0^t d\tau_1 \int^{t}_{\tau_2 } d\tau_2 \tfrac {1+(\tau_2+\xi)^2}{1+(\tau_2+\xi)^2}  \exp(-\nu (\tau_2-\tau_1+\tfrac 1 3 ((\tau_2+\xi  )^3-(\tau_1+\xi )^3)))\\
    &\le  \int_0^t d\tau_1 \int^{t}_{\tau_2 } d\tau_2 \tfrac {1+(\tau_2+\xi)^2}{1+(\tau_1+\xi)^2}  \exp(-\nu (\tau_2-\tau_1+\tfrac 1 3 ((\tau_2+\xi  )^3-(\tau_1+\xi )^3)))\\
    &\le \tfrac 1\nu \int_0^t d\tau_1  \tfrac 1{1+(\tau_1+\xi)^2}[\exp(-\nu (\tau_2-\tau_1+\tfrac 1 3 ((\tau_2+\xi  )^3-(\tau_1+\xi )^3)))]_{\tau_2=\tau_1}^{\tau_2 =t }\\
    &\le \tfrac 1{\nu \xi } 
\end{align*}
and thus 
\begin{align*}
    \vert p_2-\sqrt {\tfrac {1+(t+\xi )^2 } {1+\xi^2 }}p_{2,in}(k)\vert &\le \tfrac {\alpha^2 }{\nu \xi } \vert p_2\vert_\infty.
\end{align*}
Since $\xi\ge 2\tfrac {\alpha^2}\nu  $ we obtain 
\begin{align*}
    p_2(-1) \ge \tfrac 1 2 \sqrt {\tfrac {1+(t+\xi )^2 } {1+\xi^2 }}p_{2,in}(-1)\ge \tfrac t{2\xi}p_{2,in}(-1).
\end{align*}
Let $a(\xi ) $ be such that $\supp_\xi (a (\xi)) \subset [2\tfrac {\alpha^2}\nu,4\tfrac {\alpha^2}\nu]$ and $\int {(2+\xi^2)^{\frac N 2 }} a^2(\xi)=1$ then we deduce that for the initial data 
\begin{align*}
    p_{in} (k, \xi ) &= \textbf{1}_{k=-1} a(\xi )
\end{align*}
it holds that 
\begin{align*}
    \Vert p_{in} \Vert_{H^N}&=1,\\
    \Vert p(t)\Vert_{H^N}&\ge t\tfrac {\nu }{8\alpha^2 },\\
    \Vert p(t)\Vert_{H^{N-1}}&\ge t\tfrac {\nu^2 }{32\alpha^4},
\end{align*}
which proves \eqref{eq:isolow}. Furthermore, for all solutions such that  $p(\tau)\in H^N$ we estimate  
\begin{align*}
    \p_t \vert p\vert^2(k,\xi,t ) &\le2\tfrac {\vert t-\frac \xi k \vert }{1+\vert t-\frac \xi k \vert^2}\vert p \vert^2(k,\xi,t)
\end{align*}
and so 
\begin{align*}
    \vert p\vert^2(k,\xi,t )&\le \exp( 2\int_\tau^t \tfrac {\vert s-\frac \xi k \vert }{1+\vert s -\frac \xi k \vert^2} ds ) \vert p_{in} \vert^2(k,\xi )\\
    &\le\exp( \int_0^{t }  \tfrac { 2\tilde s  }{1+ \tilde s^2 }d\tilde s  ) \vert p_{in} \vert^2(k,\xi )\\
    &\le \langle t\rangle^4  \vert p_{in} \vert^2(k,\xi )
\end{align*}
which proves \eqref{eq:isoup}.
\end{proof}

\subsection{Nonlinear norm inflation}

We next consider the nonlinear non-resistive MHD equations in their perturbative
form around the stationary solution \eqref{eq:Couette}:
\begin{align}
  \label{NLiso}
\begin{split}
     \partial_t p_1 - \partial_x \partial_x^t \Delta^{-1}_t p_1- \alpha \partial_x p_2 &= \nu \Delta_t p_1 + \nabla^\perp_t\Lambda^{-1}_t (b\nabla_t b - v\nabla_t v ),\\
      \partial_t p_2 +\partial_x \partial_x^t \Delta^{-1}_t p_2 - \alpha \partial_x p_1 &=  \nabla^\perp_t\Lambda^{-1}_t (b\nabla_t v - v\nabla_t b ).
\end{split}
\end{align}
The following lemma establishes the norm inflation result of Proposition \ref{prop:instability}.
\begin{lemma}[Nonlinear norm inflation for the non-resistive MHD equations]
  Consider the non-resistive nonlinear MHD equations \eqref{NLiso}. Then for all $C=C(\nu ) >1 $ there exists $\eps_0>0$ such that for all $0<\eps<\eps_0 $ there exists initial data $p_{in} $ such that 
    \begin{align*}
        \Vert p_{in} \Vert_{H^N } =\eps
    \end{align*}
    and 
    \begin{align*}
        \Vert p\Vert_{L^\infty H^N } \ge \eps   C.
    \end{align*}
\end{lemma}
\begin{proof}
For the sake of contradiction we assume that there exists $\eps_0>0$ such that for all $0<\eps\le \eps_0 $ and for any choice of initial data with
$\|p_{in}\|_{H^N}=\epsilon$ it holds that 
\begin{align*}
        \Vert p\Vert_{L^\infty H^N } \le \eps  C.
\end{align*}
Our plan is to choose initial data such that for a choice of $\eps$ and $t$ we obtain a contradiction to this bound. In particular, we choose $p_{in}$ as the data of the linear instability
result, Lemma \ref{lemma:linIns}, such that the associated linear solution
$p_{lin}$ satisfies
\begin{align*}
    \Vert p_{in} \Vert_{H^N } &= \eps, \\
    \Vert p_{lin} (t)\Vert_{H^{N-1} } &\ge t \tfrac {\nu^2}{32 \alpha^4}.
\end{align*}

Let $S(\tau, t)$ be the solution operator for the linearized system. Then in view of \eqref{eq:isoup} we have the estimate
\begin{align}
    \Vert S(\tau, t)\Vert_{H^N\to H^N }\le  {\langle t\rangle^2 }.
\end{align}
Thus we deduce that
\begin{align*}
     \partial_t ( p-p_{lin})&\le L( p-p_{lin}) + NL[p]
\end{align*}
and therefore
\begin{align*}
    \Vert p-p_{lin}\Vert_{H^{N-1}}^2
    &\le \int_0^\tau \Vert S(\tau, t)\Vert_{H^N\to H^N } \Vert p-p_{lin}\Vert_{ H^{N-1 }}\Vert p\Vert_{ H^{N-1 }}\Vert   \nabla_t p\Vert_{H^{N-1}}\\
    &\lesssim \Vert p-p_{lin}\Vert_{ L^\infty H^{N-1 }}\Vert p\Vert_{L^\infty H^{N-1 }}\Vert   p\Vert_{L^\infty H^{N}} 2 \int^t_0 t \langle t \rangle^2 \\
    &\lesssim t^2\langle t\rangle^2 \eps^2 C^2 \Vert p-p_{lin}\Vert_{ H^{N-1 }}. 
\end{align*}
Finally, we obtain 
\begin{align*}
    \Vert p\Vert_{H^{N-1} }&\ge \Vert p_{lin} \Vert_{H^{N-1} }-t^2\langle t\rangle^2 \eps^2 C^2\\
    &= t \eps (  \tfrac {\nu^2 }{32\alpha^4 } - t\langle t\rangle^2 \eps C^2).
\end{align*}
This completes our proof by contradiction provided this term is large enough for a given small $\eps$ and suitable time. Indeed for the choice $\eps \le     \tfrac {\nu^8} {10^{8}C^5 \alpha^{16}} $ we obtain that at the time $t= 10^2 C  \tfrac {\alpha^4} {\nu^2 } $ it  holds that
\begin{align*}
    \Vert p\Vert_{H^{N-1} }&\ge  t  \tfrac {\nu^2 }{10^3\alpha^4 }\eps\ge C \eps .
\end{align*}
This concludes our proof of the nonlinear norm inflation and hence completes our proof of Proposition \ref{prop:instability}.
\end{proof}
The behavior of the MHD equations and, in particular, the interaction of shear flows, the magnetic field and dissipation are an area of current active research \cite{liss2020sobolev,Dolce,zhao2023asymptotic,knobel2023echoes}. 
However, prior works have focused on cases where the resistivity is at least as strong as the fluid viscosity and where thus the behavior is closely related to that of the Navier-Stokes equations.
In contrast, the non-resistive MHD equations exhibit additional instability, as for instance shown in Proposition \ref{prop:instability}.

Motivated by this dichotomy, in this article we have studied the anisotropic, partial dissipation regime
\begin{align*}
    \kappa_y=0, \ \kappa_x=\nu_x=\nu_y
\end{align*}
and the associated stability threshold in the inviscid limit.
As shown in Theorem \ref{thm:anisoThres} and highlighted in the estimates of Sections \ref{linstab}, \ref{hfa} and \ref{hfw}, this partial dissipation regime behaves qualitatively differently than both the fully dissipative case and the non-resistive case.
Moreover, our analysis crucially used the coupling of the velocity field and magnetic field induced by the underlying magnetic field, which allowed us to obtain improved estimates for the magnetic field despite the lack of the symmetry of the dissipation. 

Partial, anisotropic dissipation in the MHD equations is thus shown to give rise to distinct regimes with different (in)stability properties and demonstrates an intricate interplay of shear dynamics, magnetic interaction and anisotropic dissipation. A more complete understanding of all these regimes, the case of resistivity vanishing faster than viscosity and a characterization of the (in)stability properties of the ideal MHD equations remain exciting questions for future research.

\subsection*{Acknowledgments}
Funded by the Deutsche Forschungsgemeinschaft (DFG, German Research Foundation) – Project-ID 258734477 – SFB 1173
\appendix

\bibliography{library}
\bibliographystyle{alpha}

\end{document}